\documentclass{article}
        
\usepackage{arxiv}
\usepackage[utf8]{inputenc}
\usepackage{url}
\usepackage{amsmath}
\usepackage{amsthm}
\usepackage{mathtools}

\newcommand{\R}{\mathbb{R}}
\newcommand{\CX}{\mathcal{X}}
\newcommand{\CP}{\mathcal{P}}

\usepackage{amsfonts}
\usepackage{amssymb}
\usepackage{xcolor}
\usepackage{graphicx}
\usepackage{float}
\usepackage{enumerate}
\usepackage{enumitem}
\usepackage{multirow}

\newcommand{\norm}[1]{\left\lVert#1\right\rVert}
\newcommand{\abs}[1]{\left |#1\right |}
\usepackage{mathtools}

\usepackage{algorithmic, algorithm} 
\newtheorem{theorem}{Theorem}

\newtheorem{assumption}{Assumption}
\newtheorem{definition}{Definition}
\newtheorem{lemma}{Lemma}

\newtheorem{result}{Result}
\newtheorem{remark}{Remark}
\newtheorem{example}{Example}
\usepackage[export]{adjustbox}
\usepackage{commath}
\usepackage{hyperref}
\usepackage[capitalize,noabbrev]{cleveref}
\usepackage{todonotes}

\allowdisplaybreaks
\hypersetup{
	colorlinks,
	linkcolor=black,
	anchorcolor=black,
	citecolor=black
}
\newcommand*\diff{\mathop{}\!\mathrm{d}}

\graphicspath{ {./figure/} }
\title{Sample Complexity of Probability Divergences under Group Symmetry}
\author{  Ziyu Chen\\
  Department of Mathematics and Statistics\\
  University of Massachusetts Amherst\\
  Amherst, MA 01003,  USA \\
  \texttt{ziyuchen@umass.edu} \\
  \And
    Markos A. Katsoulakis\\
    Department of Mathematics and Statistics\\
  University of Massachusetts Amherst\\
  Amherst, MA 01003,  USA \\
  \texttt{markos@umass.edu} \\
\And
    Luc Rey-Bellet\\
    Department of Mathematics and Statistics\\
  University of Massachusetts Amherst\\
  Amherst, MA 01003,  USA \\
  \texttt{luc@umass.edu} 
  \And
    Wei Zhu\\
    Department of Mathematics and Statistics\\
  University of Massachusetts Amherst\\
  Amherst, MA 01003,  USA \\
  \texttt{weizhu@umass.edu} 
}

\date{}

\begin{document}
\maketitle

\begin{abstract}
We rigorously quantify the improvement in the sample complexity of variational divergence estimations for group-invariant distributions. In the cases of the Wasserstein-1 metric and the Lipschitz-regularized $\alpha$-divergences, the reduction of sample complexity is proportional to the group size if the group is finite. In addition to the published version at ICML 2023, our proof indeed has included the case when the group is infinite such as compact Lie groups, the convergence rate can be further improved and depends on the intrinsic dimension of the fundamental domain characterized by the scaling of its covering number. Our approach is different from that in [Tahmasebi \& Jegelka, ICML 2024] and our work also applies to asymmetric divergences, such as the Lipschitz-regularized $\alpha$-divergences. For the maximum mean discrepancy (MMD), the improvement of sample complexity is more nuanced, as it depends on not only the group size but also the choice of kernel. Numerical simulations verify our theories.
\end{abstract}

\section{Introduction}
Probability divergences provide means to measure the discrepancy between two probability distributions. They have broad applications in a variety of inference tasks, such as independence testing \cite{zhang2018large, kinney2014equitability}, independent component analysis \cite{hyvarinen2002independent}, and generative modeling \cite{goodfellow2014generative,nowozin2016f, arjovsky2017wasserstein, NIPS2017_892c3b1c, tolstikhin2018wasserstein, nietert2021smooth}.

A key task within the above applications is the computation and estimation of the divergences from finite data, which is known to be a difficult problem \cite{paninski2003estimation, gao2015efficient}. Empirical estimators based on the variational representations for the probability divergences are generally favored and widely used due to their scalability to both the data size and the ambient space dimension \cite{MINE_paper,birrell2022optimizing, nguyen2007nonparametric, nguyen2010estimating, Ruderman, sreekumar2022neural,birrell2021variational,birrell2022function,sriperumbudur2012empirical, gretton2006kernel, gretton2007kernel, gretton2012kernel, genevay2019sample}.

\begin{figure}[t]
    \centering
    \includegraphics[width=.6\columnwidth]{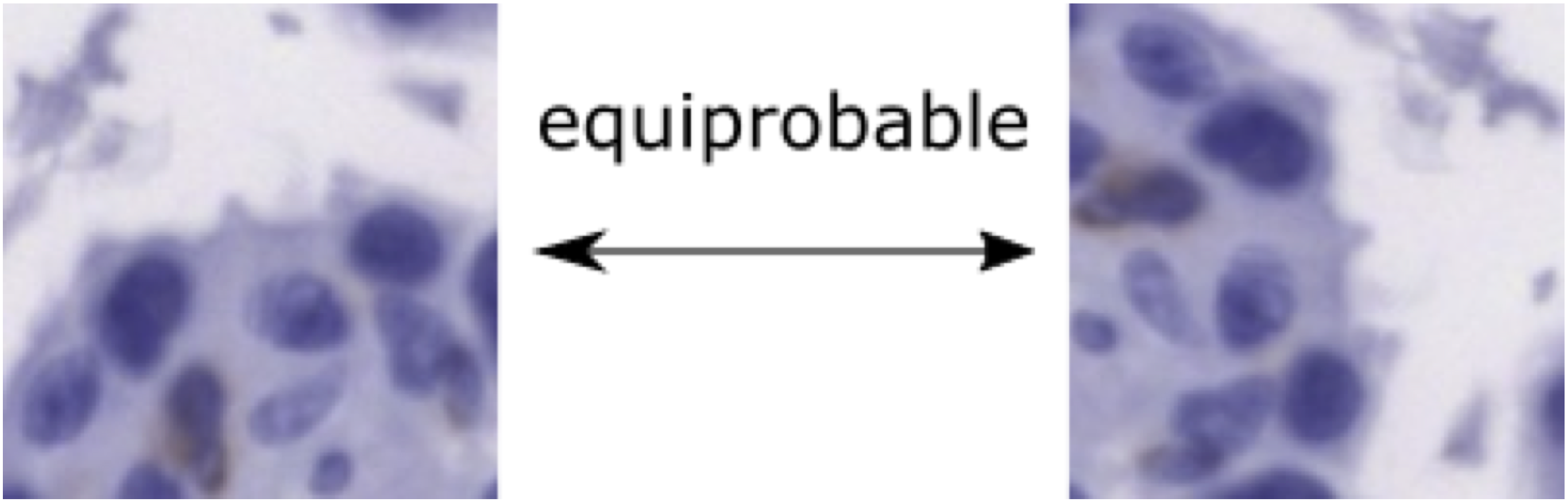}
    \caption{The distribution of the whole-slide prostate cancer images (LYSTO data set \cite{francesco_ciompi_2019_3513571}) is \textit{rotation-invariant}, i.e., an image and its rotated copies are equiprobable.}
    \label{fig:equiprobable}
\end{figure}

Empirical computation of the probability divergences and theoretical analysis on their sample complexity are typically studied without any a priori \textit{structural assumption} on the probability measures. Many distributions in real life, however, are known to have intrinsic structures, such as \textit{group symmetry}. For example, the distribution of the medical images collected without preferred orientation should be \textit{rotation-invariant}, i.e., an image is supposed to have the same likelihood as its rotated copies; see Figure~\ref{fig:equiprobable}. Such structural information could be leveraged to improve the accuracy and/or sample-efficiency for divergence estimation.

Indeed, the recent work by Birrell et al. \cite{birrell2022structure} shows that one can develop an improved variational representation for divergences between group-invariant distributions. The key idea is to reduce the test function space in the variational formula to its subset of group-invariant functions, which effectively acts as an unbiased regularization. When used in a generative adversarial network (GAN) for group-invariant distribution learning, Birrell et al. \cite{birrell2022structure} empirically show that divergence estimation/optimization based on their proposed variational representation under group symmetry leads to significantly improved sample generation, especially in the small data regime.

The purpose of this work is to rigorously quantify the performance gain of divergence estimation under group symmetry. More specifically, we analyze the reduction in sample complexity of divergence estimation in terms of the convergence of the empirical estimation. We focus, in particular, on three types of probability divergences: the Wasserstein-1 metric, the maximum mean discrepancy (MMD), and the family of Lipschitz-regularized $\alpha$-divergences; see \cref{sec:background_variational} for the exact definition. Our main results show that the reduction of sample complexity is proportional to the group size if the
group is \textit{finite}. When the group is \textit{infinite}, the convergence rate can be further improved and depends
on the intrinsic dimension of the fundamental domain characterized by the scaling of its covering
number defined in \cref{def:intrinsic_dimension}; see \cref{thm:gammaIPM_main} and \cref{thm:gammaIPM_main_infinite_group} for the Wasserstein-1 metric; \cref{thm:f-gamma_main} and \cref{thm:f-gamma_main_infinite_group} for the Lipschitz-regularized $\alpha$-divergences respectively. In the case of MMD, the reduction in sample complexity due to the group invariance is more nuanced and depends on the properties of the kernel; see Theorem~\ref{thm:mmd_main}. As a byproduct, we also establish the consistency and sample complexity for the Lipschitz-regularized $\alpha$-divergences without group symmetry, which, to the best of our knowledge, is missing in the previous literature.

The rest of the paper is organized as follows. In Section \ref{sec:related}, we review previous work related to the empirical estimation of probability divergences and group-invariant distributions. We introduce the background and the motivation in Section \ref{sec:background}. Theoretical results and numerical examples are provided in Section \ref{sec:sample_complexity}. We conclude our work in Section \ref{sec:conclusion}. Proofs and detailed statement of the theorems in Section \ref{sec:sample_complexity} are provided in Section \ref{appendix:proofs}.

\section{Related work}\label{sec:related}
\textbf{Empirical estimation of probability divergences.}
Probability divergences have been widely used, including in generative adversarial networks (GANs) \cite{arjovsky2017wasserstein,goodfellow2014generative,nowozin2016f,birrell2022structure,NIPS2017_892c3b1c},
uncertainty quantification \cite{chowdhary_dupuis_2013, DKPP}, independence determination through mutual information estimation \cite{MINE_paper}, bounding risk in probably approximately correct (PAC) learning \cite{catoni2008pac,10.1145/307400.307435,10.1145/267460.267466}, statistical mechanics and interacting particles \cite{Kipnis:99}, large deviations \cite{dupuis2011weak}, and parameter estimation \cite{Broniatowski_Keziou}.

A growing body of literature has been dedicated to the empirical estimation of divergences from finite data. Earlier works based on density estimation are known to work best for low dimensions \cite{kandasamy2015nonparametric,poczos2011nonparametric}. Recent research has shown that statistical estimators based on the variational representations of probability divergences scale better with dimensions; such studies include the KL-divergences \cite{MINE_paper}, the more general $f$-divergences \cite{birrell2022optimizing, nguyen2007nonparametric, nguyen2010estimating, Ruderman, sreekumar2022neural}, R\'enyi divergences \cite{birrell2021variational, birrell2022function}, integral probability metrics (IPMs) \cite{sriperumbudur2012empirical, gretton2006kernel, gretton2007kernel, gretton2012kernel}, and Sinkhorn divergences \cite{genevay2019sample}. Such estimators are typically constructed to compare an \textit{arbitrary} pair of probability measures without any a priori structural assumption, and are hence sub-optimal in estimating divergences between distributions with known structures, such as group symmetry.

\textbf{Group-invariant distributions.} Recent development in group-equivariant machine learning \cite{pmlr-v48-cohenc16, NEURIPS2019_b9cfe8b6, NEURIPS2019_45d6637b} has sparked a flurry of research in neural generative models for group-invariant distributions. Most of the works focus only on the guaranteed \textit{generation}, through, e.g., an equivariant normalizing-flow, of the group-invariant distributions \cite{bilovs2021scalable, boyda2021sampling, garcia2021n, kohler2019equivariant, liu2019graph, rezende2019equivariant}; the divergence computation between the generated distribution and the ground-truth target, a crucial step in the optimization pipeline, however, does not leverage their group-invariant structure. Equivariant GANs for group-invariant distribution learning have also been proposed by modifying the inner loop of discriminator update through either data-augmentation \cite{zhao2020differentiable} or constrained optimization within a subspace of group-invariant discriminators \cite{EquivariantGAN}; the theoretical justification of such procedures, as well as the resulting performance gain, have been explained by Birrell et al. \cite{birrell2022structure} as an improved estimation of variational divergences under group symmetry via an unbiased regularization. The exact \textit{quantification} of the improvement, in terms of reduction in sample complexity, is however still missing; this is the main focus of this work.

\section{Background and motivation}
\label{sec:background}
\subsection{Variational divergences and probability metrics}
\label{sec:background_variational}
Let $\mathcal{X}$ be a measurable space, and $\CP(\CX)$ be the set of probability measures on $\CX$. A map $D:\CP(\CX)\times \CP(\CX) \to [0, \infty]$ is called a \textit{divergence} on $\CP(\CX)$ if
\begin{align}
    D(P, Q) = 0 \iff P = Q \in \CP(\CX),
\end{align}
hence providing a notion of ``distance" between probability measures. Many probability divergences of interest can be formulated using a variational representation
\begin{align}
\label{eq:variational_divergence}
    D(P, Q) = \sup_{\gamma\in\Gamma}H(\gamma; P, Q),
\end{align}
where $\Gamma\subset \mathcal{M}(\CX)$ is a space of test functions, $\mathcal{M}(\CX)$ is the set of measurable functions on $\CX$, and $H:\mathcal{M}(\CX)\times\mathcal{P}(\CX)\times \CP(\CX)\to [-\infty, \infty]$ is some objective functional. Through  suitable choices of $H(\gamma; P, Q)$ and $\Gamma$, formula~\eqref{eq:variational_divergence} includes many divergences and probability metrics. Below we list two specific classes of examples.

\textbf{(a) $\Gamma$-Integral Probability Metrics ($\Gamma$-IPMs).} Given $\Gamma\subset \mathcal{M}_b(\CX)$, the space of bounded measurable functions on $\CX$, the $\Gamma$-IPM between $P$ and $Q$ is defined as
\begin{align}
    \label{eq:IPM}
    D^\Gamma(P, Q) \coloneqq \sup_{\gamma\in\Gamma}\left\{E_P[\gamma] - E_Q[\gamma]\right\}.
\end{align}
Some prominent examples of the $\Gamma$-IPMs include the Wasserstein-1 metric, the total variation metric, the Dudley metric, and the maximum mean discrepancy (MMD) \cite{muller1997integral,sriperumbudur2012empirical}. Our work, in particular, focuses on the following two specific IPMs.
\begin{itemize}
    \item The Wasserstein-1 metric, $W(P, Q)\coloneqq D^{\text{Lip}_L(\CX)}(P, Q)$, i.e., 
    \begin{align}
        \label{eq:Wasserstein}
        W(P, Q) \coloneqq \sup_{\gamma\in\text{Lip}_L(\CX)}\{E_P[\gamma]- E_Q[\gamma]\},
    \end{align}
    where $\text{Lip}_L(\CX)$ is the space of $L$-Lipschitz functions on $\mathcal{\CX}$. We note that the normalizing factor $L^{-1}$ has been omitted from the formula.
    \item The \textit{maximum mean discrepancy}, $\text{MMD}(P, Q)\coloneqq D^{B_{\mathcal{H}}}(P, Q)$, i.e.,
    \begin{align}
        \label{eq:MMD}
        \text{MMD}(P, Q) \coloneqq \sup_{\gamma\in B_{\mathcal{H}}}\{E_P[\gamma]- E_Q[\gamma]\},
    \end{align}
    where $B_\mathcal{H}$ is the unit ball of some reproducing kernel Hilbert space (RKHS) $\mathcal{H}$ on $\mathcal{X}$.    
\end{itemize}

\textbf{(b) $(f, \Gamma)$-divergences.} Let $f : [0, \infty) \to \R$ be convex and lower semi-continuous, with $f(1)=0$ and $f$ strictly convex at $x = 1$. Given $\Gamma\subset \mathcal{M}_b(\CX)$ that is closed under the shift transformations $\gamma \mapsto \gamma +\nu, \nu\in\R$, the $(f, \Gamma)$-divergence introduced by Birrell et al. \cite{birrell2020f} is defined as
\begin{align}
\label{eq:f_gamma_divergence}
    D_{f}^{\Gamma}(P\| Q) = \sup_{\gamma\in\Gamma}\{E_{P}[\gamma]-E_Q[f^*(\gamma)]\},
\end{align}
where $f^*$ denotes the Legendre transform of $f$. Formula~\eqref{eq:f_gamma_divergence} includes, as a special case when $\Gamma = \mathcal{M}_b(\CX)$, the widely known class of $f$-divergences, with notable examples such as the Kullback-Leibler (KL) divergence \cite{kullback1951information}, the total variation distance, the Jensen-Shannon divergence, the $\chi^2$-divergence, the Hellinger distance, and more generally the family of $\alpha$-divergences \cite{nowozin2016f}. Of particular interest to us is the class of the \textit{Lipschitz-regularized} $\alpha$-divergences, 
\begin{align}
\label{eq:alpha_divergence}
    D_{f_\alpha}^\Gamma (P\| Q), ~~\Gamma = \text{Lip}_{L}(\mathcal{X}), ~f_\alpha(x)=\frac{x^\alpha-1}{\alpha(\alpha-1)},
\end{align}
where $\alpha > 0$ and $\alpha\neq 1$ is a positive parameter.

An important observation that will be useful in one of our results,  Theorem~\ref{thm:f-gamma_main}, is that $D_{f_\alpha}^\Gamma$ admits an equivalent representation, which writes
\begin{equation}\label{eq:falphagammanew}
    D_{f_\alpha}^{\Gamma}(P\| Q) = \sup_{\gamma\in\Gamma,\nu\in\mathbb{R}}\{E_{P}[\gamma+\nu]-E_{Q}[f_\alpha^*(\gamma+\nu)]\}
\end{equation}
due to the invariance of $\Gamma=\text{Lip}_{L}(\mathcal{X})$ under the shift map $\gamma\mapsto\gamma+\nu$ for $\nu\in\mathbb{R}$.
\subsection{Empirical estimation of variational divergences}

Given {i.i.d.} samples $X = \{x_1, x_2, \cdots, x_m\}$ and $Y = \{y_1, y_2, \cdots, y_n\}$, respectively, from two unknown probability measures $P, Q\in\CP(\CX)$, it is often of interest---in applications such as two-sample testing \cite{bickel1969distribution, gretton2006kernel,gretton2012kernel, cheng2021kernel} and independence testing \cite{gretton2007kernel, gretton2012kernel, zhang2018large, kinney2014equitability}---to estimate the divergence between $P$ and $Q$ \cite{sriperumbudur2012empirical,birrell2021variational,nguyen2007nonparametric,nguyen2010estimating}. For variational divergences $D^\Gamma(P, Q)$ and $D^\Gamma_f(P\|Q)$ in the form of \eqref{eq:IPM} and \eqref{eq:f_gamma_divergence}, their empirical estimators can naturally be given by
\begin{align}
\label{eq:estimator_general_ipm}
  & D^\Gamma(P_m, Q_n) = \sup_{\gamma\in\Gamma}\left\{\sum_{i=1}^m \frac{\gamma(x_i)}{m} - \sum_{i=1}^n\frac{\gamma(y_i)}{n}\right\},\\ \label{eq:estimator_general_f_gamma}
  & D^\Gamma_f(P_m\|Q_n)= \sup_{\gamma\in\Gamma}\left\{\sum_{i=1}^m \frac{\gamma(x_i)}{m} - \sum_{i=1}^n\frac{f^*(\gamma(y_i))}{n}\right\}
\end{align}
where $P_m = \frac{1}{m}\sum_{i=1}^m\delta_{x_i}$ and $Q_n = \frac{1}{n}\sum_{j=1}^n\delta_{y_j}$ represent the empirical distributions of $P$ and $Q$, respectively.

The consistency and sample complexity of the empirical estimators $W(P_m, Q_n)$ and $\text{MMD}(P_m, Q_n)$ in the form of \eqref{eq:estimator_general_ipm}
for, respectively, the Wasserstein-1 metric \eqref{eq:Wasserstein} and MMD \eqref{eq:MMD} between two \textit{general distributions} $P, Q\in\CP(\CX)$ have been well studied \cite{sriperumbudur2012empirical,gretton2012kernel}. However,  for probability measures with special structures, such as \textit{group symmetry}, one can potentially obtain a divergence estimator with substantially improved sample complexity as empirically observed by Birrell et al. \cite{birrell2022structure}. We provide, in the following section, a brief review of group-invariant distributions and the improved variational representations for probability  divergences under group symmetry, which serves as a motivation and foundation for our theoretical analysis in \cref{sec:sample_complexity}.

\subsection{Variational divergences under group symmetry}

A \textit{group} is a set $\Sigma$ equipped with a group product satisfying the axioms of associativity, identity, and invertibility. Given a group $\Sigma$ and a set $\CX$, a map $\theta:\Sigma \times \CX\to\CX$ is called a \textit{group action on $\CX$} if $\theta_\sigma\coloneqq \theta(\sigma, \cdot): \CX\to\CX$ is an automorphism on $\CX$ for all $\sigma\in\Sigma$, and $\theta_{\sigma_2}\circ \theta_{\sigma_1} = \theta_{\sigma_2\cdot \sigma_1}$, $\forall \sigma_1, \sigma_2\in\Sigma$. By convention, we will abbreviate $\theta(\sigma, x)$ as $\sigma x$ throughout the paper.

A function $\gamma:\CX\to \R$ is called \textit{$\Sigma$-invariant} if $\gamma\circ \theta_\sigma = \gamma, \forall \sigma\in\Sigma$. Let $\Gamma$ be a set of measurable functions $\gamma:\mathcal{X}\to\mathbb{R}$; its subset, $\Gamma_{\Sigma}$, of $\Sigma$-invariant functions is defined as
\begin{equation}
\label{eq:invariant_function_space}
    \Gamma_{\Sigma} \coloneqq \{\gamma\in\Gamma:\gamma\circ \theta_\sigma=\gamma,\forall \sigma\in\Sigma\}.
\end{equation}
On the other hand, a probability measure $P\in \CP(\CX)$ is called \textit{$\Sigma$-invariant} if $P = (\theta_\sigma)_* P, \forall \sigma\in \Sigma$, where $(\theta_\sigma)_* P\coloneqq P\circ (\theta_\sigma)^{-1}$ is the push-forward measure of $P$ under $\theta_\sigma$. We denote the set of all $\Sigma$-invariant distributions on $\CX$ as $\CP_\Sigma(\CX)\coloneqq \{P\in\CP(\CX): P ~\text{is}~ \Sigma\text{-invariant}\}$.

Finally, for a compact Hausdorff topological group $\Sigma$ \cite{folland1999real}, we define two \textit{symmetrization operators}, $S_\Sigma:\mathcal{M}_b(\CX)\to\mathcal{M}_b(\CX)$ and $S^\Sigma:\CP(\CX)\to\CP(\CX)$, on functions and probability measures, respectively, as follows
\begin{align}
\label{eq:symmetrization_function}
    &S_\Sigma[\gamma](x)  \coloneqq \int_\Sigma \gamma(\sigma x)\mu_\Sigma(d\sigma),~\forall \gamma\in\mathcal{M}_b(\CX)\\
\label{eq:symmetrization_measure}
    &E_{S^\Sigma[P]} \gamma \coloneqq E_PS_\Sigma[\gamma], ~\forall P\in\CP(\CX), \forall \gamma \in\mathcal{M}_b(\CX)
\end{align}
where $\mu_\Sigma$ is the unique Haar probability measure on $\Sigma$. The operators $S_\Sigma[\gamma]$ and $S^\Sigma[P]$ can be intuitively understood, respectively,  as ``averaging" the function $\gamma$ or ``spreading" the probability mass $P$ across the group orbits in $\CX$; one can easily verify that they are \textit{projection operators} onto the corresponding invariant subsets $\Gamma_\Sigma\subset \Gamma$ and $\CP_\Sigma(\CX)\subset \CP(\CX)$ \cite{birrell2022structure}.

The main result by Birrel et al. \cite{birrell2022structure}, which we summarize in Result~\ref{thm:sp-gan-main-result}, is that for $\Sigma$-invariant distributions, the function space $\Gamma$ in the variational formulae \eqref{eq:IPM} and \eqref{eq:f_gamma_divergence} can be reduced to its invariant subset $\Gamma_\Sigma\subset\Gamma$.
\begin{result}[paraphrased from \cite{birrell2022structure}]
\label{thm:sp-gan-main-result}
    If $S_\Sigma[\Gamma]\subset \Gamma$ and $P, Q\in\CP(X)$, then
    \begin{align}
        & D^\Gamma(S^\Sigma[P], S^\Sigma[Q]) = D^{\Gamma_\Sigma}(P, Q),\\
        & D_f^\Gamma(S^\Sigma[P]\| S^\Sigma[Q]) = D^{\Gamma_\Sigma}_f(P\| Q), 
    \end{align}
    where $D^\Gamma(P, Q)$ and $D_f^\Gamma(P\|Q)$ are given by \eqref{eq:IPM} and \eqref{eq:f_gamma_divergence}. In particular, if $P, Q\in\CP_\Sigma(\CX)$ are $\Sigma$-invariant, then
    \begin{align*}
        D^\Gamma(P, Q) = D^{\Gamma_\Sigma}(P, Q), \quad D_f^\Gamma(P\| Q) = D^{\Gamma_\Sigma}_f(P\| Q).
    \end{align*}
\end{result}
Result~\ref{thm:sp-gan-main-result} motivates a potentially more sample-efficient way to estimate the divergences $D^\Gamma(P, Q)$ and $D_f^\Gamma(P\| Q)$ between $\Sigma$-invariant distributions $P, Q\in\CP(\CX)$ using
\begin{align}
\label{eq:estimator_invariant_ipm}
  & D^{\Gamma_\Sigma}(P_m, Q_n) = \sup_{\gamma\in\Gamma_\Sigma}\left\{\sum_{i=1}^m \frac{\gamma(x_i)}{m} - \sum_{i=1}^n\frac{\gamma(y_i)}{n}\right\},\\ \label{eq:estimator_invariant_f_gamma}
  & D^{\Gamma_\Sigma}_f(P_m\|Q_n)= \sup_{\gamma\in\Gamma_\Sigma}\left\{\sum_{i=1}^m \frac{\gamma(x_i)}{m} - \sum_{i=1}^n\frac{f^*(\gamma(y_i))}{n}\right\}.
\end{align}
Compared to Eq.~\eqref{eq:estimator_general_ipm} and \eqref{eq:estimator_general_f_gamma}, the estimators \eqref{eq:estimator_invariant_ipm} and \eqref{eq:estimator_invariant_f_gamma} have the benefit of optimizing over a reduced space $\Gamma_\Sigma\subset \Gamma$ of test functions, effectively acting as an \textit{unbiased regularization}, and their efficacy has been empirically observed by Birrell et al. \cite{birrell2022structure} in neural generation of group-invariant distributions with substantially improved data-efficiency. However, the theoretical understanding of the performance gain is still lacking.

The purpose of this work is to rigorously quantify the improvement in sample complexity of the divergence estimations \eqref{eq:estimator_invariant_ipm} and \eqref{eq:estimator_invariant_f_gamma} for group-invariant distributions. To contextualize the idea, we will focus our analysis on three specific types of probability divergences, the Wasserstein-1 metric \eqref{eq:Wasserstein}, the MMD \eqref{eq:MMD}, and the Lipschitz-regularized $\alpha$ divergence \eqref{eq:f_gamma_divergence}\eqref{eq:alpha_divergence} between $\Sigma$-invariant $P, Q\in\CP_\Sigma(\CX)$,
\begin{align}
    \label{eq:estimator_wasserstein_invariant}
    &W(P, Q) = W^\Sigma(P, Q) \approx W^\Sigma(P_m, Q_n),\\
    \label{eq:estimator_mmd_invariant}
    &\text{MMD}(P, Q) = \text{MMD}^\Sigma(P, Q) \approx \text{MMD}^\Sigma(P_m, Q_n)\\
    \label{eq:estimator_alpah_invariant}
    &D^\Gamma_{f_\alpha}(P\|Q) = D_{f_\alpha}^{\Gamma_\Sigma}(P\|Q)\approx D^{\Gamma_\Sigma}_{f_\alpha}(P_m\|Q_n),
\end{align}
where
\begin{align}
    & W^\Sigma(P, Q) \coloneqq D^{[\text{Lip}_L(\CX)]_\Sigma}(P, Q),\\
    & \text{MMD}^\Sigma(P, Q)\coloneqq D^{[B_{\mathcal{H}}]_\Sigma}(P, Q),
\end{align}
and the definition of $D_{f_\alpha}^{\Gamma_\Sigma}(P\|Q)$ is given by Equations~\eqref{eq:f_gamma_divergence}, \eqref{eq:alpha_divergence} and \eqref{eq:invariant_function_space}.

\subsection{Further notations and assumptions}
\label{sec:assumptions_notations}
For the rest of the paper, we assume the measurable space $\CX\subset \R^D$ is a bounded subset of $\R^D$ equipped with the Euclidean metric $\|\cdot\|_2$. The group is denoted by $\Sigma$. The Haar measure $\mu_\Sigma$ is thus a uniform probability measure over $\Sigma$, and the symmetrization $S_\Sigma[\gamma]$ [Eq.~\eqref{eq:symmetrization_function}] is an average of $\gamma$ over the group orbit. We next introduce the concept of \textit{fundamental domain} in the following definition.
\begin{definition}
    \label{def:fundamental_domain}
    A subset $\CX_0\subset \CX$ is called a \textit{fundamental domain} of $\CX$ under the $\Sigma$-action if for each $x\in\CX$, there exists a unique $x_0\in\CX$ such that $x = \sigma x_0$ for some $\sigma\in\Sigma$.
\end{definition}
\begin{figure}[ht]	  
    \begin{center}
\centerline{    \includegraphics[width=.4\columnwidth]{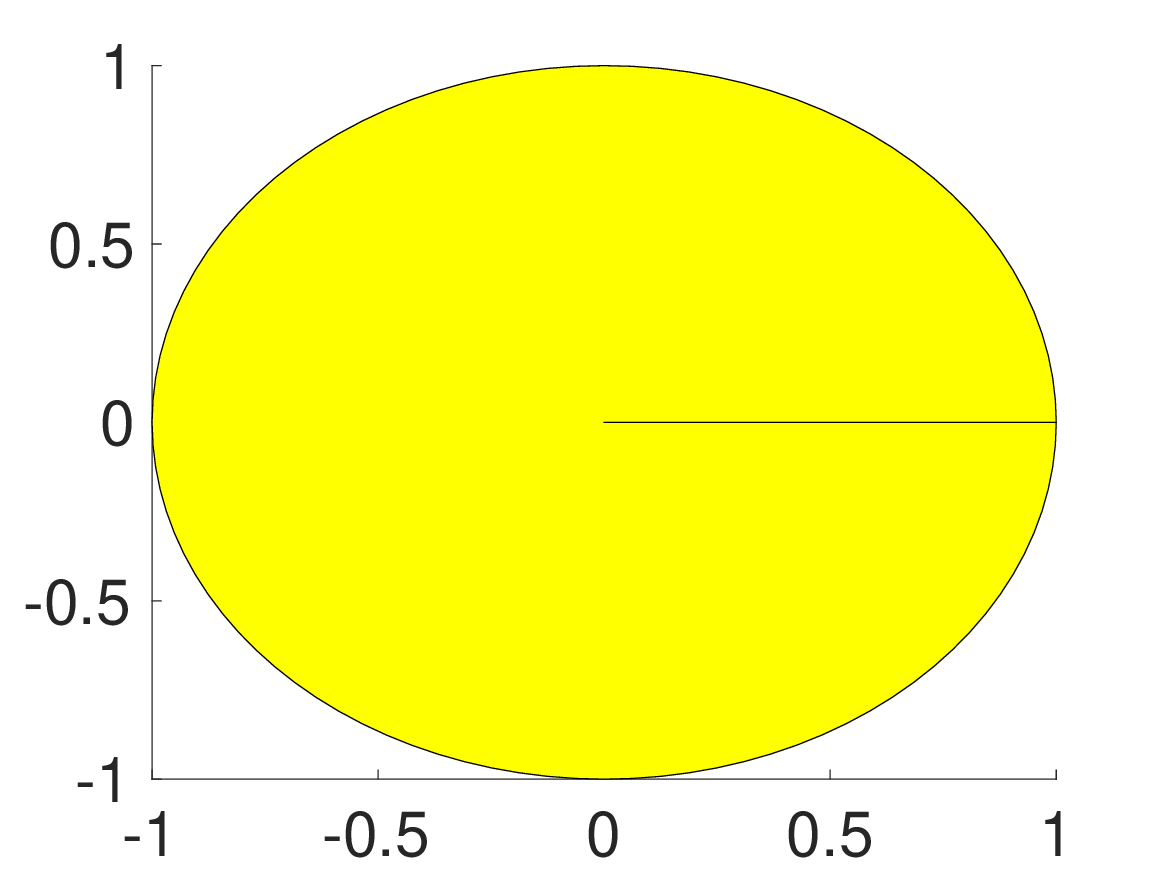}
    \includegraphics[width=.4\columnwidth]{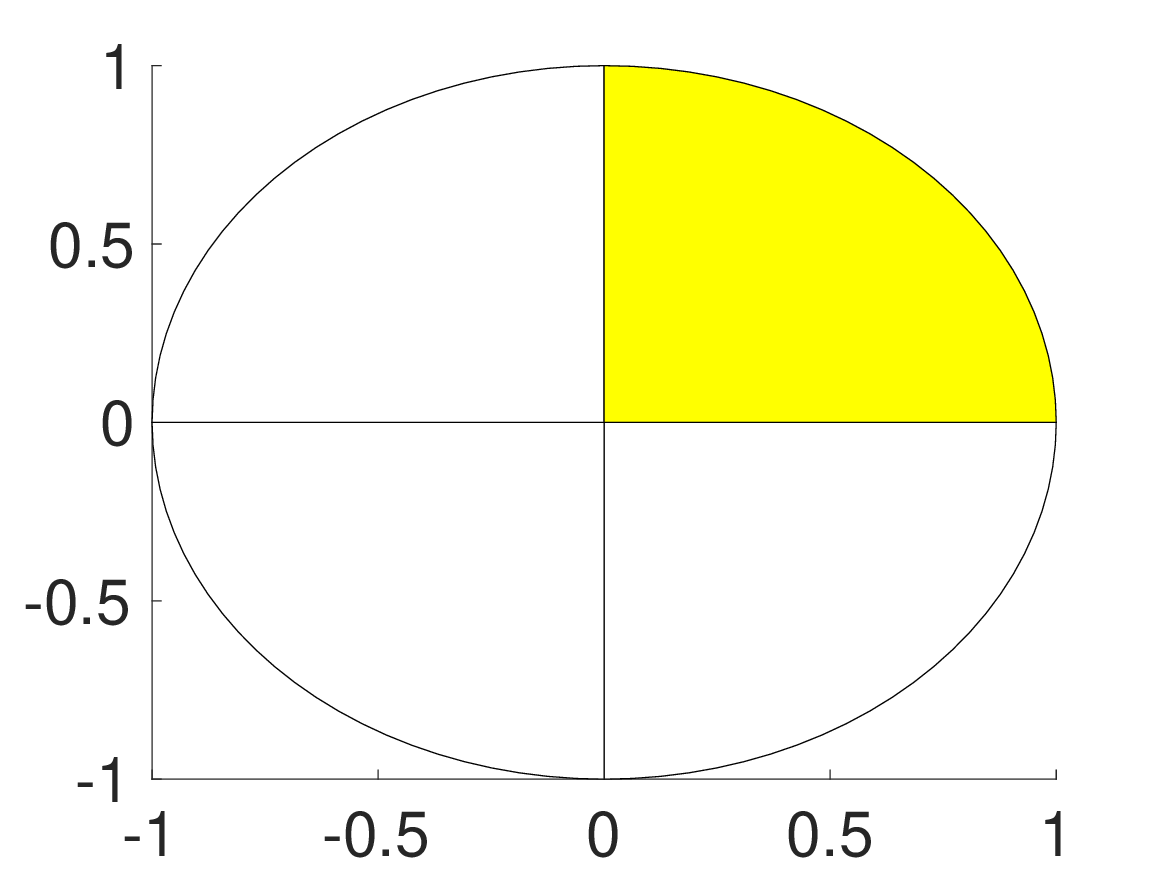}}
\centerline{\includegraphics[width=.4\columnwidth]{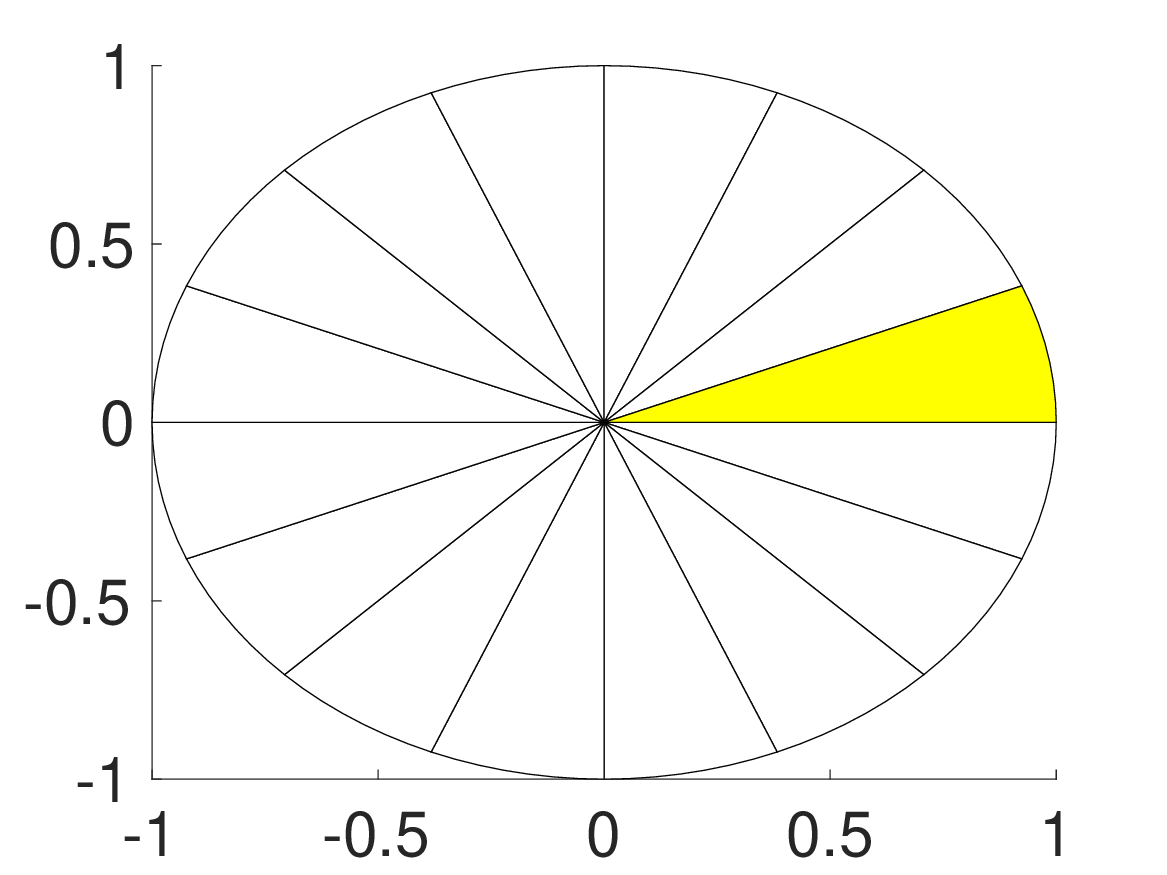}
    \includegraphics[width=.4\columnwidth]{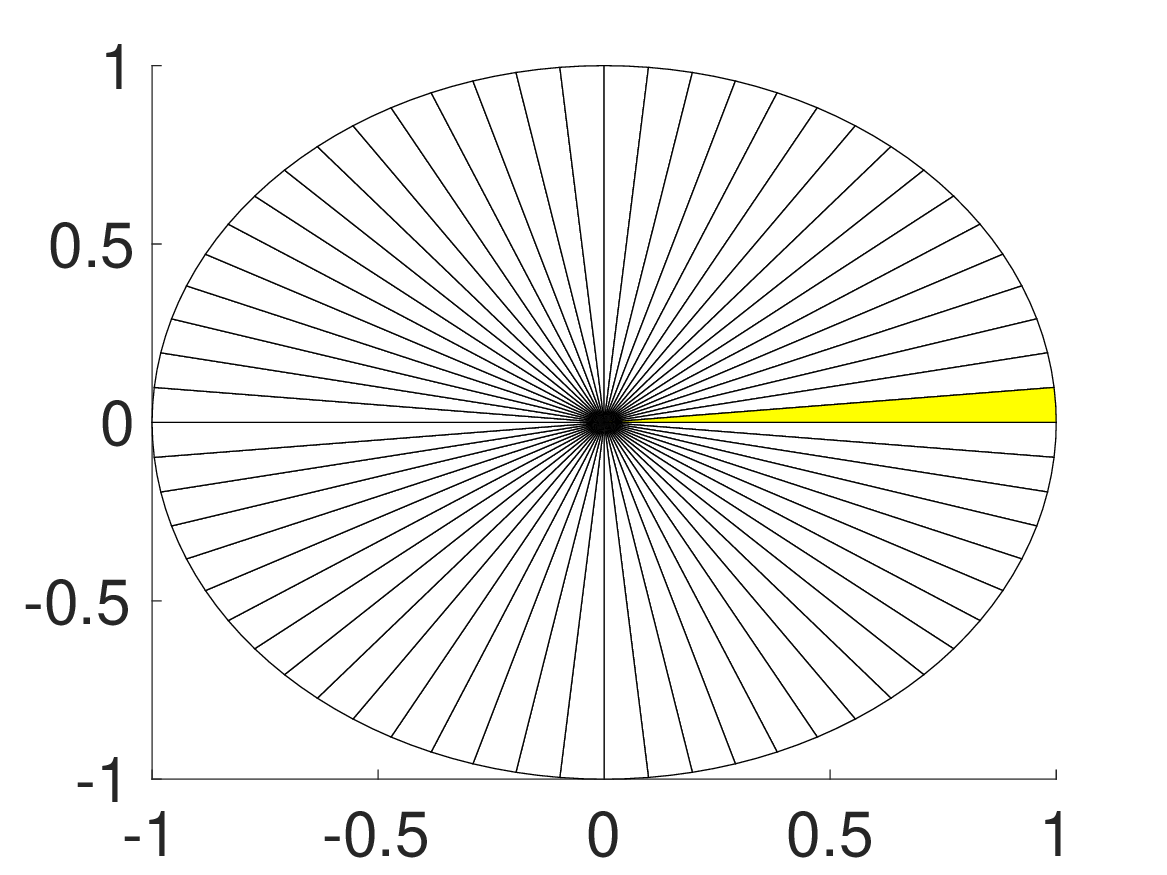}}
    \caption{The unit disk $\CX\subset \R^2$ with the action of the (discrete) rotation groups $\Sigma=C_n$, $n = 1, 4, 16, 64$. The fundamental domain $\mathcal{X}_0$ for each $C_n$ is filled with yellow color.}
	\label{fig:disk}
    \end{center}
    \vskip -0.2in
\end{figure}
Figure~\ref{fig:disk} displays an example where $\CX$ is the unit disk in $\R^2$, and $\Sigma=C_n, n=1, 4, 16, 64,$ are the discrete rotation groups acting on $\CX$; the fundamental domain $\CX_0$ for each $\Sigma=C_n$ is filled with yellow color. We note that the choice of the fundamental domain $\CX_0$ is not unique. We will slightly abuse the notation $\CX = \Sigma\times \CX_0$ to denote $\CX_0$ being a fundamental domain of $\CX$ under the $\Sigma$-action. We define $T_0:\mathcal{X}\to\mathcal{X}_0$
\begin{align}\label{def:quotientmap}
T_0(x) \coloneqq y\in\CX_0, ~\text{if}~y=\sigma x ~\text{for some}~\sigma\in\Sigma,
\end{align}
i.e., $T_0$ maps $x\in\CX$ to its unique orbit representative in $\CX_0$. In addition, we denote by $P_{\mathcal{X}_0}\in \CP(\CX_0)$ the distribution on the fundamental domain $\mathcal{X}_0$ induced by a $\Sigma$-invariant distribution $P\in\CP_\Sigma(\CX)$ on $\CX$; that is,
\begin{equation}
\label{eq:fundamental_domain_measure}
dP_{\mathcal{X}_0}(x) = \int_{y\in\CX: y=\sigma x, \sigma\in\Sigma}dP(x), x\in\CX_0.
\end{equation}
The \textit{diameter} of $\CX\subset \R^d$ is defined as
\begin{align}
    \text{diam}(\CX) = \sup_{x, y\in \CX}\|x-y\|_2.
\end{align}
Finally, part of our results in Section~\ref{sec:sample_complexity} relies heavily on the concept of \textit{covering numbers} which we define below.
\begin{definition}[Covering number] Let $(\CX,\rho)$ be a metric space. A subset $S\subset \CX$ is called a $\delta$-cover of $\mathcal{X}$ if for any $x\in \mathcal{X}$ there is an $s\in S$ such that $\rho(s,x)\leq\delta$. The $\delta$-covering number of $\mathcal{X}$ is defined as
\[
\mathcal{N}(\mathcal{X},\delta,\rho):=\min\left\{\abs{S}:S \text{ is a } \delta\text{-cover of } \mathcal{X}\right\}.
\]
When $\rho(x, y) = \|x-y\|_2$ is the Euclidean metric in $\R^D$, we abbreviate $\mathcal{N}(\mathcal{X},\delta,\rho)$ as $\mathcal{N}(\mathcal{X},\delta)$.
\end{definition}
Following the notion of capacity dimension from \cite{kegl2002intrinsic}, we define the intrinsic dimension of a set as follows.
\begin{definition}\label{def:intrinsic_dimension}
    The intrinsic dimension of $\mathcal{X}\subset\mathbb{R}^D$, denoted by $dim(\mathcal{X})$ is defined as
    \begin{equation}
        dim(\mathcal{X})\coloneqq-\lim_{\delta\to0^+}\frac{\ln\mathcal{N}(\mathcal{X},\delta)}{\log\delta}.
    \end{equation}
\end{definition}
For example, if $\mathcal{X}\subset\mathbb{R}^D$ contains an open set of $\mathbb{R}^D$, then $\dim(\mathcal{X})=D$; if $\mathcal{X}$ is a $d$-dimensional submanifold of $\mathbb{R}^D$, then $\dim(\mathcal{X})=d$.

\section{Sample complexity under group invariance}
\label{sec:sample_complexity}
In this section, we outline our main results for the sample complexity of divergence estimation under group invariance. In particular, we focus on three cases: the Wasserstein-1 metric \eqref{eq:estimator_wasserstein_invariant}, the MMD \eqref{eq:estimator_mmd_invariant} and the $(f_\alpha,\Gamma)$-divergence \eqref{eq:estimator_alpah_invariant}. While the convergence rate in the bounds for the Wasserstein-1 metric and the $(f_\alpha,\Gamma)$-divergence depends on the dimension of the ambient space, that for the MMD case does not. In all the numerical experiments, for simplicity, we choose $X = \{x_1, x_2, \cdots, x_m\}$ and $Y = \{y_1, y_2, \cdots, y_n\}$ to sample from the same $\Sigma$-invariant distribution $P=Q$ for easy visualization and clear benchmark.

\subsection{Wasserstein-1 metric, $W(P, Q)$}
\label{sec:ipm}
In this section, we set $\Gamma = \text{Lip}_L(\mathcal{X})$ to be the set of $L$-Lipschitz functions on $\CX$; see Eq.~\eqref{eq:Wasserstein}. We further assume that the $\Sigma$-actions on $\mathcal{X}$ is $1$-Lipschitz, i.e., $\|\sigma x-\sigma y\|_2\le \|x-y\|_2, \forall \sigma\in\Sigma, \forall x, y\in\CX$, so that $S_\Sigma[\Gamma]\subset \Gamma$ (see Lemma \ref{lemma:symmetrizeGamma} for a proof). Due to Result~\ref{thm:sp-gan-main-result}, we have $W(P, Q) = W^\Sigma(P, Q)$ for $\Sigma$-invariant probability measures $P, Q\in\CP_\Sigma(\CX)$.

To convey the main message, we provide a summary of our result in Theorem~\ref{thm:gammaIPM_main} for the sample complexity under group invariance for the Wasserstein-1 metric. The detailed statement and the technical assumption of the theorem as well as its proof are deferred to \cref{appendix:gammaIPM}. Readers are referred to \cref{sec:background} for the notations.

\begin{theorem}[Finite groups]
\label{thm:gammaIPM_main}
Let $\mathcal{X} = \Sigma\times\mathcal{X}_0$ be a bounded subset of $\mathbb{R}^D$ equipped with the Euclidean distance, $\abs{\Sigma}<\infty$ and $\dim(\mathcal{X})=d$. Suppose $P, Q\in\CP_\Sigma(\CX)$ are $\Sigma$-invariant distributions on $\mathcal{X}$. If the number $m,n$ of samples drawn from $P$ and $Q$ are sufficiently large, then we have with high probability, 

1) when $d\geq2$, for any $s>0$ ,
\begin{align}
\nonumber
&\abs{W(P,Q)-W^{\Sigma}(P_m,Q_n)}\\
\label{eq:d_2_ipm}
&\quad\leq C\left[\left(\frac{1}{\abs{\Sigma}m}\right)^{\frac{1}{d+s}}+\left(\frac{1}{\abs{\Sigma}n}\right)^{\frac{1}{d+s}}\right],
\end{align}
where $C>0$ depends only on $d, s$ and $\CX$, and is independent of $m$ and $n$;

2) for $d=1$, we have
\begin{align}
\nonumber
&\abs{W(P,Q)-W^{\Sigma}(P_m,Q_n)}\\
&\quad\leq C\cdot\text{\normalfont diam}(\CX_0)\left(\frac{1}{\sqrt{m}}+\frac{1}{\sqrt{n}}\right),
\end{align}
where $C>0$ is an absolute constant independent of $\CX, \CX_0, m$ and $n$.
\end{theorem}
When $\Sigma$ is infinite and $\dim(\mathcal{X}_0)=d^*$, we have the following result with improved convergence rates.
\begin{theorem}[Infinite groups]\footnote{Compared to the work in \cite{tahmasebisample} for continuous groups, our proof relies on the covering numbers, while \cite{tahmasebisample} applies the Weyl's law on Riemannian manifolds without using covering number argument.}
\label{thm:gammaIPM_main_infinite_group}
Let $\mathcal{X} = \Sigma\times\mathcal{X}_0$ be a bounded subset of $\mathbb{R}^D$ equipped with the Euclidean distance and $\dim(\mathcal{X}_0)=d^*$. Suppose $P, Q\in\CP_\Sigma(\CX)$ are $\Sigma$-invariant distributions on $\mathcal{X}$. Then we have with high probability, 

1) when $d^*\geq2$, for any $s>0$ ,
\begin{align}
\nonumber
&\abs{W(P,Q)-W^{\Sigma}(P_m,Q_n)}\\
\label{eq:d_2_ipm}
&\quad\leq C\left[\left(\frac{\text{vol}(\mathcal{X}_0)}{m}\right)^{\frac{1}{d^*+s}}+\left(\frac{\text{vol}(\mathcal{X}_0)}{n}\right)^{\frac{1}{d^*+s}}\right],
\end{align}
where $C>0$ depends only on $d, s$ and $\CX$, and is independent of $m$ and $n$, and $\text{vol}(\mathcal{X}_0)$ is the volume of $\mathcal{X}_0$;

2) for $d^*=1$, we have
\begin{align}
\nonumber
&\abs{W(P,Q)-W^{\Sigma}(P_m,Q_n)}\\
&\quad\leq C\cdot\text{\normalfont diam}(\CX_0)\left(\frac{1}{\sqrt{m}}+\frac{1}{\sqrt{n}}\right),
\end{align}
where $C>0$ is an absolute constant independent of $\CX, \CX_0, m$ and $n$.
\end{theorem}
Furthermore, if $\mathcal{X}$ is a $d$-dimensional connected submanifold, and $\Sigma$ is a compact Lie group acting locally smoothly on $\mathcal{X}$, then $d^*=d-\dim(\Sigma)$, where $\dim(\Sigma)$ is the dimension of a principal orbit (i.e., the maximal
dimension among all orbits) by Theorem IV 3.8 in \cite{bredon1972introduction}. This recovers the bound derived in \cite{tahmasebisample}.

\textit{Sketch of the proof.} 
Using the group invariance and the map $T_0$ defined in \eqref{def:quotientmap}, we can transform the i.i.d. samples on $\CX$ to i.i.d. samples on $\CX_0$, which are effectively sampled from $P_{\CX_0}$ and $Q_{\CX_0}$ [cf.~Eq.~\eqref{eq:fundamental_domain_measure}]. Hence the supremum after applying the triangle inequality to the error \eqref{eq:d_2_ipm} can be taken over $L$-Lipschitz functions defined on the fundamental domain $\CX_0$, i.e., $\text{Lip}_L(\CX_0)$, instead of over the original space $\text{Lip}_L(\CX)$. We further demonstrate in \cref{lemma:infinitybound} that the supremum can be taken over an even smaller function space
\begin{align}
\label{eq:F_0}
    \mathcal{F}_0 = \{\gamma\in\text{Lip}_L(\mathcal{X}_0):\norm{\gamma}_\infty\leq M\}\subset \text{Lip}_L(\CX_0),
\end{align}
with some uniformly bounded $L^\infty$-norm $M$ due to the translation-invariance of $\gamma$ in definition \eqref{eq:Wasserstein}. Using the Dudley's entropy integral \cite{bartlett2017spectrally}, the error can be bounded in terms of the metric entropy of $\mathcal{F}_0$,
\begin{align}
    \inf_{\alpha>0} \left\{8\alpha+\frac{24}{\sqrt{m}}\int_{\alpha}^{M}\sqrt{\ln\mathcal{N}(\mathcal{F}_0,\delta,\norm{\cdot}_{\infty})}\diff{\delta}\right\}.
\end{align}
For $d\geq2$, we establish the relations between the metric entropy, $\ln\mathcal{N}(\mathcal{F}_0,\delta,\norm{\cdot}_{\infty})$, of $\mathcal{F}_0$ and the covering numbers of $\CX_0$ and $\CX$ via \cref{lemma:coverbycover} and \cref{lemma:ratioofcovering}:
\begin{equation}\label{eq:coverbycover}
    \ln\mathcal{N}(\mathcal{F}_0,\delta,\norm{\cdot}_{\infty})\leq\mathcal{N}(\mathcal{X}_0,\frac{c_2\delta}{L})\ln(\frac{c_1M}{\delta}),
\end{equation}
\begin{equation}\label{eq:ratiobygroupsize}
    \frac{\mathcal{N}(\mathcal{X}_0,\delta)}{\mathcal{N}(\mathcal{X},\delta)} \leq\frac{1}{\abs{\Sigma}}, ~\text{for small enough}~\delta,
\end{equation}
which yields a factor in terms of the group size $\abs{\Sigma}$ in Eq.~\eqref{eq:d_2_ipm} when $\Sigma$ is finite. The dominant term of the bound based on the singularity of the entropy integral at $\alpha=0$ is shown in Eq.~\eqref{eq:d_2_ipm}. 
For $d=1$, the entropy integral is not singular at the origin, and we bound the covering number of $\mathcal{F}_0$ by $\text{diam}(\CX_0)$ instead. The probability bound is from the application of the McDiarmid's inequality. For \cref{thm:gammaIPM_main_infinite_group}, the result is direct using \cref{eq:coverbycover} with the condition $\mathcal{N}(\mathcal{X}_0,\delta)\lesssim\delta^{-d^*}$ implied by \cref{def:intrinsic_dimension} without resorting to \cref{eq:ratiobygroupsize}.
\begin{remark}
    Even though we present in Theorem~\ref{thm:gammaIPM_main} only the dominant terms showing the rate of convergence for the  estimator, our result for sample complexity is actually non-asymptotic. See Theorem~\ref{thm:gammaIPM} in \cref{appendix:gammaIPM} for a complete description of the result.
\end{remark}
\begin{remark}
    When $|\Sigma|= 1$, i.e., no group symmetry is leveraged in the divergence estimation, our result reduces to the case considered in, e.g., \cite{sriperumbudur2012empirical}, for general distributions $P, Q\in\CP_\Sigma(\CX)= \CP(\CX)$.
\end{remark}
\begin{remark}
\label{remark:ratio_main_d_2}
In the case for $d\geq2$, the $s>0$ in \cref{thm:gammaIPM_main} means the rate can be arbitrarily close to $-\frac{1}{d}$. If we further assume that $\CX_0$ is connected, then the bound can be improved to $\abs{W(P,Q)-W^{\Sigma}(P_m,Q_n)}\leq C\left[\left(\frac{1}{\abs{\Sigma}m}\right)^{\frac{1}{2}}\ln m+\left(\frac{1}{\abs{\Sigma}n}\right)^{\frac{1}{2}}\ln n\right]$ for $d=2$, and $\abs{W(P,Q)-W^{\Sigma}(P_m,Q_n)}\leq C\left[\left(\frac{1}{\abs{\Sigma}m}\right)^{\frac{1}{d}}+\left(\frac{1}{\abs{\Sigma}n}\right)^{\frac{1}{d}}\right]$ for $d\geq 3$, without the dependence of $s$, which matches the rate in \cite{fournier2015rate}. See Remark~\ref{rmk:coverbycover} after Lemma~\ref{lemma:coverbycover}. Similar improvement holds for \cref{thm:gammaIPM_main_infinite_group} in $d^*$. 
\end{remark}
\begin{remark}
\label{remark:ratio_main}
    The factor $\text{diam}(\CX_0)$ in the case for $d=1$ is not necessarily directly related to the group size $\abs{\Sigma}$. We refer to \cref{example:wss1d} below and its explanation in \cref{rmk:wss-1d} for cases where we can achieve a factor of $\abs{\Sigma}^{-1}$ in the convergence bound.
\end{remark}

\begin{example}\label{example:wss1d}
Let $\mathcal{X} = [0,1)$ and $\Gamma = \text{Lip}_L\left([0,1)\right)$, i.e., $d=1$. We consider the $\Sigma$-actions on $\CX$ generated by the translation $x \mapsto (x+\frac{1}{r})\mod1$, where $r = 1, 4, 16, 64, 256$, so that $|\Sigma| = r$ is the group size. We draw samples $x_i\sim P=Q\in \CP_\Sigma(\CX)$ on $\CX$ in the following way: $x_i = r^{-1}\xi^{1/3}_i +\eta_i$, where $\xi_i$ are i.i.d. uniformly distributed random variables on $[0,1)$ and $\eta_i$ take values over $\{0,\frac{1}{r}, \dots,\frac{r-1}{r}\}$ with equal probabilities. One can easily verify that $P=Q$ are indeed $\Sigma$-invariant. The numerical results for the empirical estimation of $W(P, Q)= 0$ using $W^\Sigma(P_n, Q_n)$ with different group size $|\Sigma| = r$, $r = 1, 4, 16, 64, 256$, are shown in the left panel of \cref{fig:wasserstein}. One can clearly observe a significant improvement of the estimator as the group size $|\Sigma|$ increases. Furthermore, the right panel of \cref{fig:wasserstein} displays the ratios between the adjacent curves, all of which converge to 4, which is the ratio between the consecutive group size. This matches our calculation in \cref{rmk:wss-1d}; see also \cref{remark:ratio_main}.
\end{example}

\begin{example}\label{example:wss2d}
We let $\mathcal{X} = \mathbb{R}^2$, i.e., $d=2$. The probability distributions $P = Q$ are the mixture of 8 Gaussians centered at $\left(\cos(\frac{2\pi r}{8}),\sin(\frac{2\pi r}{8})\right)$, $r=0,1,\dots,7$, with the same covariance. The distribution has $C_8$-rotation symmetry, but we pretend that it is only $C_1, C_2$ and $C_4$; that is, the $\Sigma$ used in the empirical estimation $W^{\Sigma}(P_m,Q_n)$ does not reflect the entire invariance structure. Even though in this case the domain $\CX$ is unbounded, which is beyond our theoretical assumptions, we can still see in \cref{fig:wasserstein2d} that as we increase the group size $|\Sigma|$ in the computation of $W^{\Sigma}(P_m,Q_n)$, fewer samples are needed to reach the same accuracy level in the approximation. The ratios between adjacent curves in this case are slight above the predicted value $\sqrt{2}\approx 1.414$ according to our theory (see \cref{remark:ratio_main_d_2}), suggesting that the complexity bound could be further improved. For instance, in \cite{sriperumbudur2012empirical}, a logarithmic correction term can be revealed for $d=2$ after a more thorough analysis.
\end{example}

\begin{figure}[ht]
\vskip 0.2in
    \begin{center}
\centerline{\includegraphics[width=.4\columnwidth]{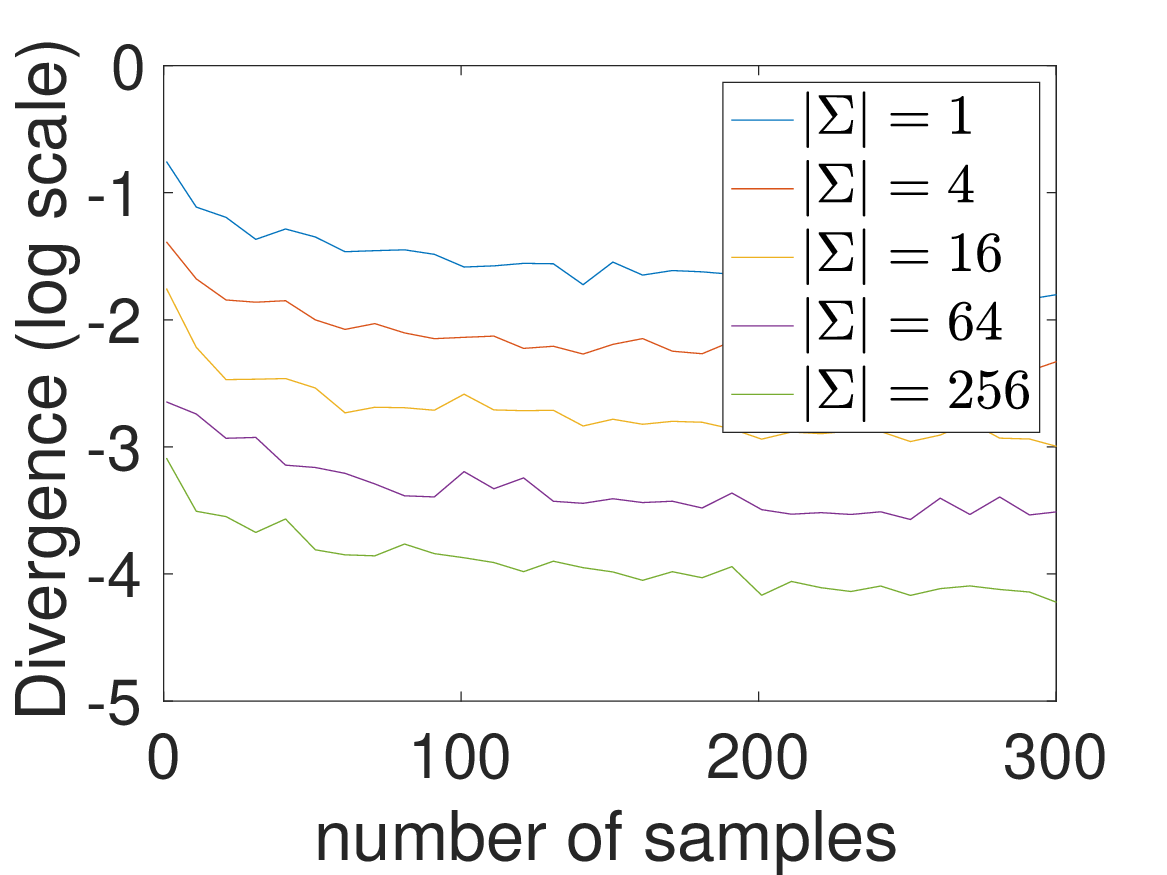}
\includegraphics[width=.4\columnwidth]{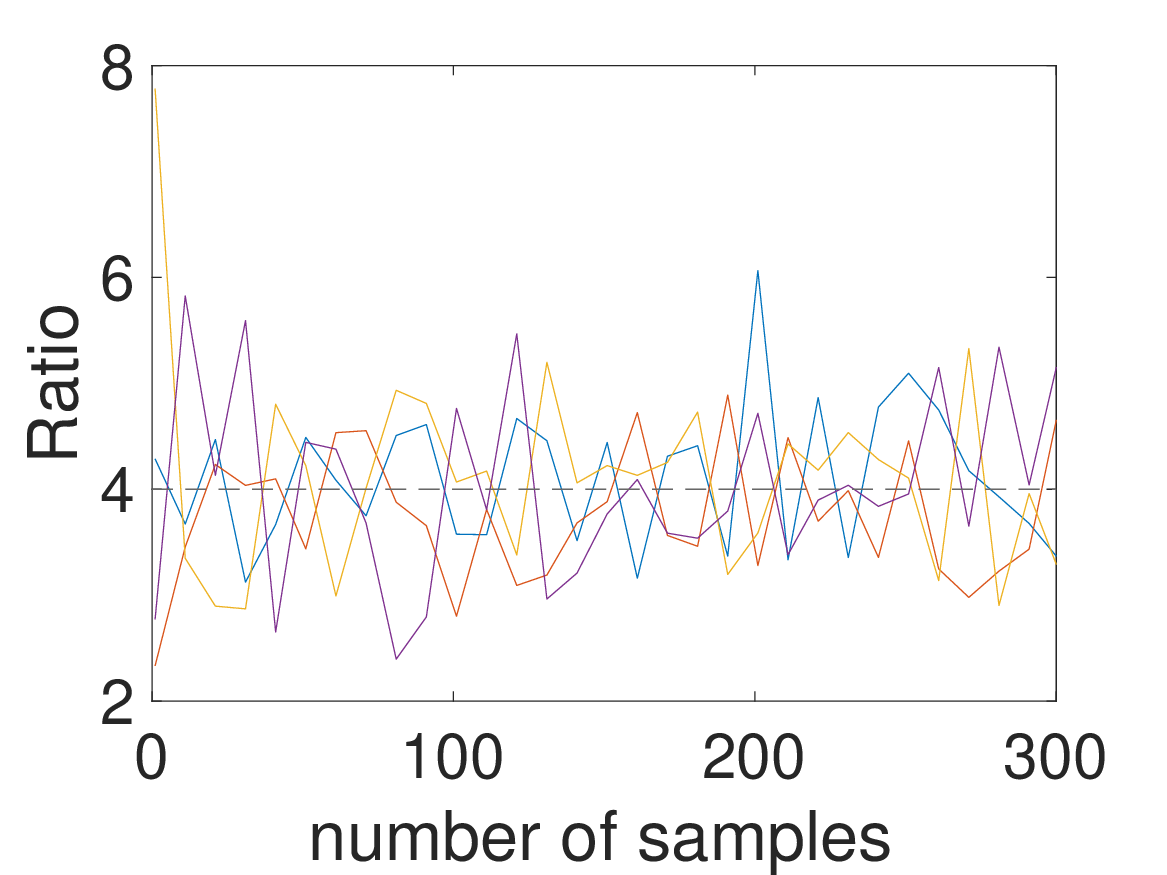}}
    \caption{Left: the Wasserstein-1 distance with different group sizes on $[0,1)$, averaged over 10 replicas. Right: the ratio of the average of the Wasserstein-1 distance between different group sizes: $\abs{\Sigma} =1$ over $\abs{\Sigma} =4$ (blue), $\abs{\Sigma} =4$ over $\abs{\Sigma} =16$ (red), $\abs{\Sigma} =16$ over $\abs{\Sigma} = 64$ (orange), $\abs{\Sigma} =64$ over $\abs{\Sigma} =256$ (purple). The black horizontal dashed line refers to the ratio equal to 4, which is the value theoretically predicted in \cref{thm:gammaIPM_main} for $d=1$. See \cref{example:wss1d} and \cref{rmk:wss-1d} for the detail.}
    \label{fig:wasserstein}
    \end{center}
    \vskip -0.2in
\end{figure}

\begin{figure}[ht]
\vskip 0.2in
    \begin{center}
\centerline{\includegraphics[width=.4\columnwidth]{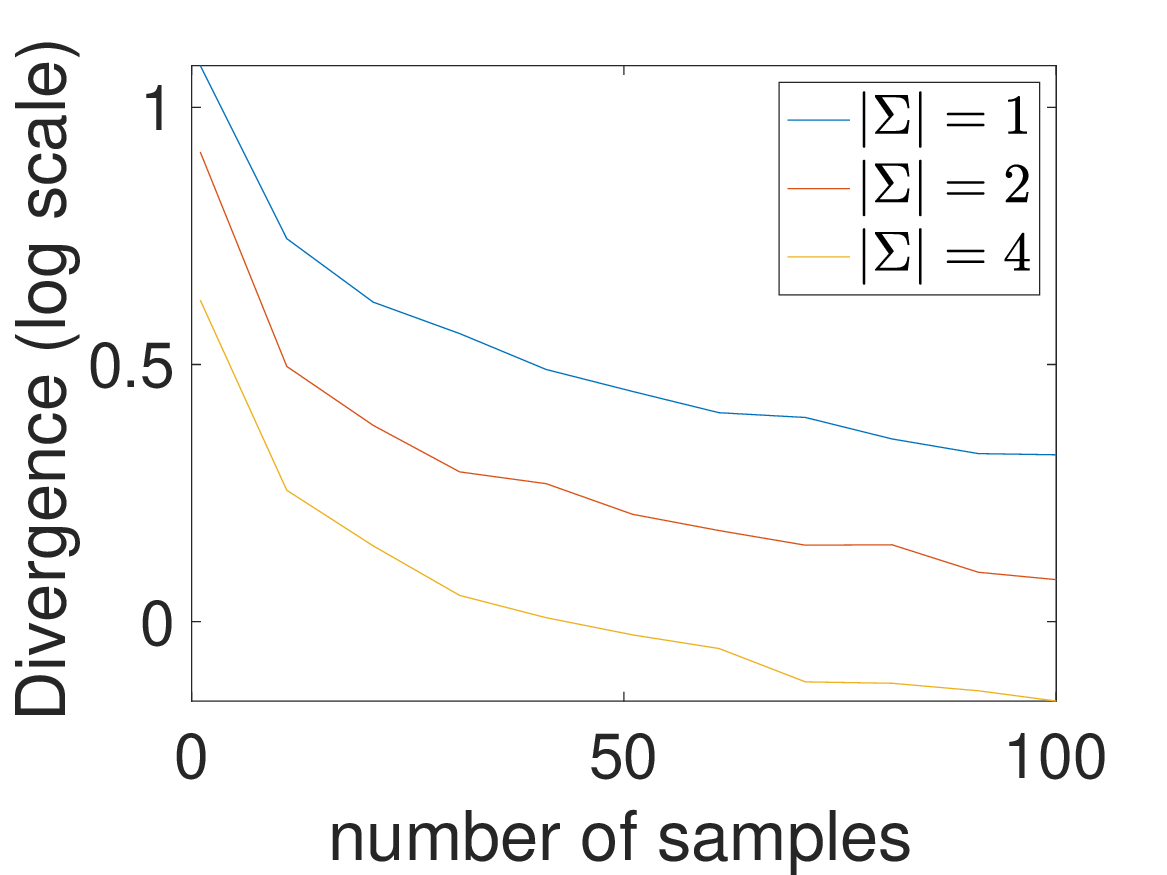}
\includegraphics[width=.4\columnwidth]{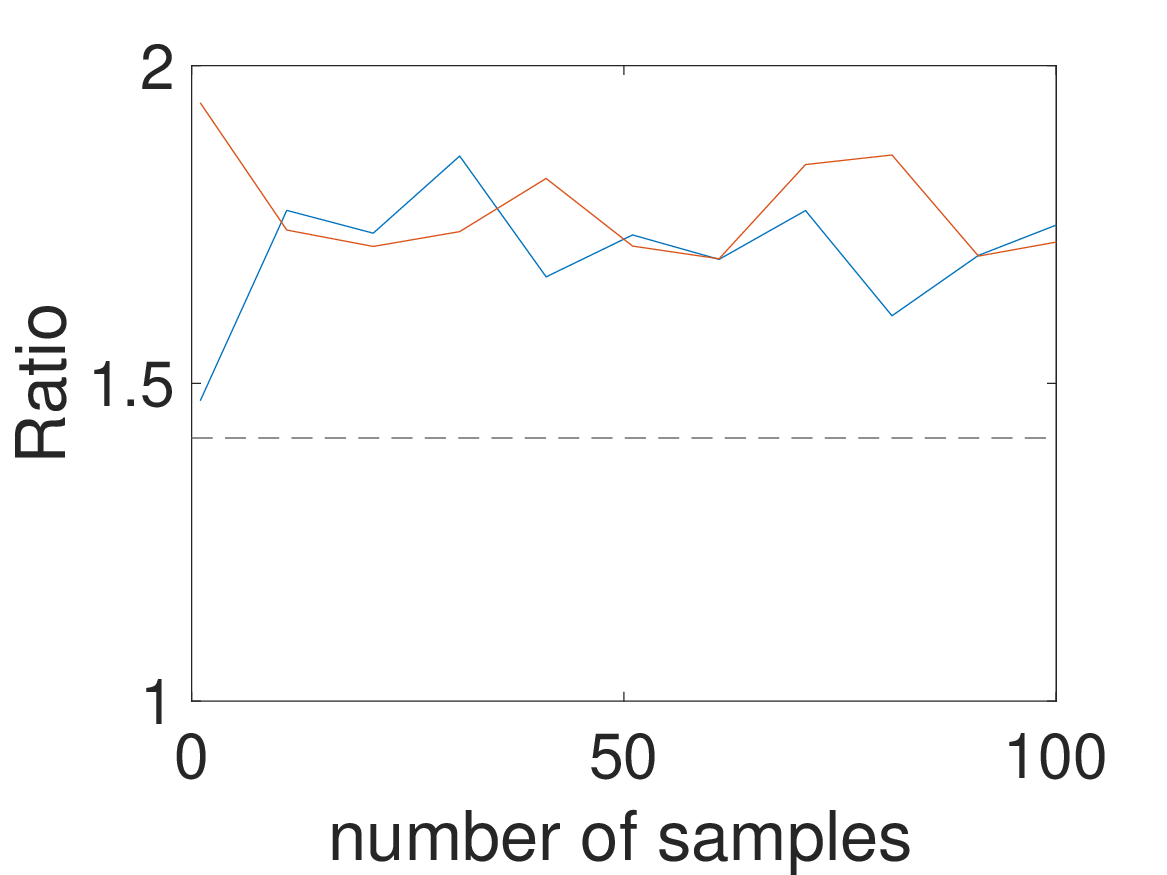}}
    \caption{Left: The Wasserstein-1 distance assuming different group sizes in $\mathbb{R}^2$, averaged over 10 replicas. Right: the ratio of the average of the Wasserstein-1 distance between different group sizes: $\abs{\Sigma} =1$ over $\abs{\Sigma} =2$ (blue), $\abs{\Sigma} =2$ over $\abs{\Sigma} =4$ (red). The black horizontal dashed line refers to the ratio equal to $\sqrt{2}$, which is the value theoretically predicted in \cref{thm:gammaIPM_main} for $d=2$. The ratios are slight above the reference line, suggesting that the complexity bound could be further improved. See \cref{example:wss2d} and \cref{remark:ratio_main_d_2} for the detail.}
    \label{fig:wasserstein2d}
    \end{center}
    \vskip -0.2in
\end{figure}

\subsection{Lipschitz-regularized $\alpha$-divergence,  $D_{f_\alpha}^\Gamma(P\|Q)$}
The space $\Gamma$ in this section is always set to $\Gamma = \text{Lip}_L(\CX)$; see Eq.~\eqref{eq:alpha_divergence}. We only consider $\alpha > 1$, as the case when $0<\alpha<1$ can be derived in a similar manner. For $\alpha>1$, the Legendre transform of $f_\alpha$, which is defined in \eqref{eq:alpha_divergence}, is
\[
f_\alpha^*(y) = \left(\alpha^{-1}(\alpha-1)^{\frac{\alpha}{\alpha-1}}y^{\frac{\alpha}{\alpha-1}}+\frac{1}{\alpha(\alpha-1)}\right)\mathbf{1}_{y>0}.
\]

We provide a theorem for the sample complexity for the $(f_\alpha,\Gamma)$-divergence under group invariance, whose detailed statement and proof can be found in \cref{appendix:falphagamma}. We note that this is a new sample complexity result for the $(f_\alpha,\Gamma)$-divergence even without the group structure, which is still missing in the literature. 
\begin{theorem}[Finite groups]
\label{thm:f-gamma_main}
Let $\mathcal{X} = \Sigma\times\mathcal{X}_0$ be a subset of $\mathbb{R}^D$ equipped with the Euclidean distance, $\abs{\Sigma}<\infty$ and $\dim(\mathcal{X})=d$. Let $f_\alpha(x)=\frac{x^\alpha-1}{\alpha(\alpha-1)}$, $\alpha>1$ and $\Gamma = \text{Lip}_L(\mathcal{X})$. Suppose $P$ and $Q$ are $\Sigma$-invariant distributions on $\mathcal{X}$. If the number of samples $m,n$ drawn from $P$ and $Q$ are sufficiently large, we have with high probability,

1) when $d\geq2$, for any $s>0$ ,
\begin{align}
\nonumber
&\abs{D_{f_\alpha}^{\Gamma}(P\| Q)-D_{f_\alpha}^{\Gamma_{\Sigma}}(P_m\|Q_n)}\\
&\leq C_1\left(\frac{1}{\abs{\Sigma}m}\right)^{\frac{1}{d+s}}+C_2\left(\frac{1}{\abs{\Sigma}n}\right)^{\frac{1}{d+s}},
\end{align}
where $C_1$ depends only on $d,s$ and $\CX$; $C_2$ depends only on $d,s$, $\CX$ and $\alpha$; both $C_1$ and $C_2$ are independent of $m$ and $n$;

2) for $d=1$, we have
\begin{align}
\nonumber
&\abs{D_{f_\alpha}^{\Gamma}(P\| Q)-D_{f_\alpha}^{\Gamma_{\Sigma}}(P_m\|Q_n)}\\
&\quad\leq \text{\normalfont diam}(\CX_0)\left(\frac{C_1}{\sqrt{m}}+\frac{C_2}{\sqrt{n}}\right),
\end{align}
where $C_1$ and $C_2$ are independent of $\CX_0$, $m$ and $n$; $C_2$ depends on $\alpha$.
\end{theorem}
When $\Sigma$ is infinite and $\dim(\mathcal{X}_0)=d^*$ for some $d^*<d$, we have the following result with improved convergence rates.
\begin{theorem}[Infinite groups]
\label{thm:f-gamma_main_infinite_group}
Let $\mathcal{X} = \Sigma\times\mathcal{X}_0$ be a subset of $\mathbb{R}^D$ equipped with the Euclidean distance and $\dim(\mathcal{X}_0)=d^*$. Let $f_\alpha(x)=\frac{x^\alpha-1}{\alpha(\alpha-1)}$, $\alpha>1$ and $\Gamma = \text{Lip}_L(\mathcal{X})$. Suppose $P$ and $Q$ are $\Sigma$-invariant distributions on $\mathcal{X}$. If the number of samples $m,n$ drawn from $P$ and $Q$ are sufficiently large, we have with high probability,

1) when $d\geq2$, for any $s>0$ ,
\begin{align}
\nonumber
&\abs{D_{f_\alpha}^{\Gamma}(P\| Q)-D_{f_\alpha}^{\Gamma_{\Sigma}}(P_m\|Q_n)}\\
&\leq C_1\left(\frac{\text{vol}(\mathcal{X}_0)}{m}\right)^{\frac{1}{d^*+s}}+C_2\left(\frac{\text{vol}(\mathcal{X}_0)}{n}\right)^{\frac{1}{d^*+s}},
\end{align}
where $C_1$ depends only on $d,s$ and $\CX$; $C_2$ depends only on $d,s$, $\CX$ and $\alpha$; both $C_1$ and $C_2$ are independent of $m$ and $n$;

2) for $d=1$, we have
\begin{align}
\nonumber
&\abs{D_{f_\alpha}^{\Gamma}(P\| Q)-D_{f_\alpha}^{\Gamma_{\Sigma}}(P_m\|Q_n)}\\
&\quad\leq \text{\normalfont diam}(\CX_0)\left(\frac{C_1}{\sqrt{m}}+\frac{C_2}{\sqrt{n}}\right),
\end{align}
where $C_1$ and $C_2$ are independent of $\CX_0$, $m$ and $n$; $C_2$ depends on $\alpha$.
\end{theorem}
\textit{Sketch of the proof.} The idea is similar to the proof of \cref{thm:gammaIPM_main}. The only difference is that we need to tackle the $f_\alpha^*(\gamma)$ term separately, since it is not translation-invariant in $\gamma$. Using the equivalent form \eqref{eq:falphagammanew}, we can obtain a different Lipschitz constant associated with $f_\alpha^*$, as well as a different $L^\infty$ bound $M$ than that in Eq.~\eqref{eq:F_0} by \cref{lemma:falphainfinitybound}. This results in the $\alpha$ dependence for $C_2$.
\begin{remark}[Compact Lie groups]
    Similar to \cref{thm:gammaIPM_main} and \cref{thm:gammaIPM_main_infinite_group}, $s$ in the bound can be removed in \cref{thm:f-gamma_main} and \cref{thm:f-gamma_main_infinite_group} when $\mathcal{X}_0$ is connected. Furthurmore, if $\mathcal{X}$ is a $d$-dimensional connected submanifold, and $\Sigma$ is a compact Lie group acting locally smoothly on $\mathcal{X}$, then $d^*=d-\dim(\Sigma)$, where $\dim(\Sigma)$ is the dimension of a principal orbit (i.e., the maximal
dimension among all orbits) by Theorem IV 3.8 in \cite{bredon1972introduction}.
\end{remark}
\subsection{Maximum mean discrepancy, $\text{MMD}(P, Q)$}
Though one can utilize the results on the covering number of the unit ball of a reproducing kernel Hilbert space, e.g. \cite{zhou2002covering, kuhn2011covering}, to derive the sample complexity bounds that depend on the dimension $d$, we provide a dimension-independent bound as in \cite{gretton2012kernel} without the use of the covering numbers. In the MMD case, we let $B_\mathcal{H}$ represent the unit ball in some reproducing kernel Hilbert space (RKHS) $\mathcal{H}$ on $\mathcal{X}$; see Eq.~\eqref{eq:MMD}. In addition, we make the following assumptions for the kernel $k(x,y)$.
\begin{assumption}
\label{assump:kernel}
The kernel $k(x, y)$ for $\mathcal{H}$ satisfies
\begin{itemize}
\item $k(x,y)\geq0$ and $k(\sigma(x),\sigma(y))=k(x,y),\forall{\sigma}\in\Sigma, x,y\in\mathcal{X}$;

\item Let $K:=\max_{x,y\in\mathcal{X}}k(x,y)$, then $k(x,y)=K$ if and only if $x=y$;

\item There exists $c_{\Sigma, k}\in(0,1)$ such that for any $\sigma\in\Sigma$ and $\sigma$ is not the identity element and $x\in\mathcal{X}_0$, we have $k(\sigma x,x)\leq c_{\Sigma, k} K$. 
\end{itemize}
\end{assumption}
Intuitively, the third condition in Assumption~\ref{assump:kernel} suggests uniform decay of the kernel on the group orbits. See \cref{rmk:kernel} and \cref{example:1} for more details and a related example.

From Lemma C.1 in \cite{birrell2022structure}, we know $S_{\Sigma}[\Gamma]\subset\Gamma$ by the first assumption. Below is an abbreviated result for the sample complexity for the MMD, whose detailed statement and proof can be found in \cref{appendix:MMD}.
\begin{theorem}\label{thm:mmd_main}
Let $\mathcal{X} = \Sigma\times\mathcal{X}_0$ and $\abs{\Sigma}<\infty$. $\mathcal{H}$ is a RKHS on $\mathcal{X}$ whose kernel satisfies Assumption~\ref{assump:kernel}. Suppose $P$ and $Q$ are $\Sigma$-invariant distributions on $\mathcal{X}$. Then for $m,n$ sufficiently large, we have with high probability,
\begin{align}
\nonumber
    &\abs{\text{\normalfont MMD}(P,Q)-\text{\normalfont MMD}^{\Sigma}(P_m,Q_n)}\\
    &\quad< O\left(C_{\Sigma, k}\left(\frac{1}{\sqrt{m}} + \frac{1}{\sqrt{n}}\right)\right),
\end{align}
where $C_{\Sigma, k} = \sqrt{\frac{1+c_{\Sigma, k}(\abs{\Sigma}-1)}{\abs{\Sigma}}}$, and $c_{\Sigma, k}$ is the constant in Assumption~\ref{assump:kernel}.
\end{theorem}
\textit{Sketch of the proof.} Based on \cref{thm:sp-gan-main-result}, we use the equality $\text{MMD}^{\Sigma}(P_m,Q_n) = \text{MMD}(S^\Sigma[P_m],S^\Sigma[Q_n])$ to expand the divergence over all the orbit elements. The error bound is controlled in terms of the Rademacher average, whose supremum is attained at some known witness function due to the structure of the RKHS using \cref{lemma:mmd-Rademacher}. Since the Rademacher average is estimated without covering numbers, the rate is independent of the dimension $d$. Then we use the decay of the kernel to obtain the bound.

\begin{remark}
When $\abs{\Sigma} = 1$, the proof is reduced to that in \cite{sriperumbudur2012empirical}.
\end{remark}

\begin{remark}\label{rmk:kernel}
Unlike the cases for the Wasserstein metric and the  Lipschitz-regularized $\alpha$-divergence in \cref{thm:gammaIPM_main} and \cref{thm:f-gamma_main}, the improvement of the sample complexity under group symmetry for MMD (measured by $C_{\Sigma, k}$ in \cref{thm:mmd_main}) depends on not only the group size $|\Sigma|$ but also the kernel $k(x, y)$. For a fixed $\mathcal{X}$ and kernel $k(x, y)$, simply increasing the group size $\abs{\Sigma}$ does not necessarily lead to a reduced sample complexity beyond a certain threshold; see the first four subfigures in \cref{fig:MMD-gaussian}. However, we show in \cref{example:1} below that, by adaptively picking a suitable kernel $k$ depending on the group size $\abs{\Sigma}$,
one can obtain an improvement in sample complexity by $C_{\Sigma,k}\approx\frac{1}{\sqrt{\abs{\Sigma}}}$ for arbitrarily large $\abs{\Sigma}$.
\end{remark}

\begin{figure}[htbp]
    \vskip 0.2in
    \begin{center}
\centerline{\includegraphics[width=.4\columnwidth]{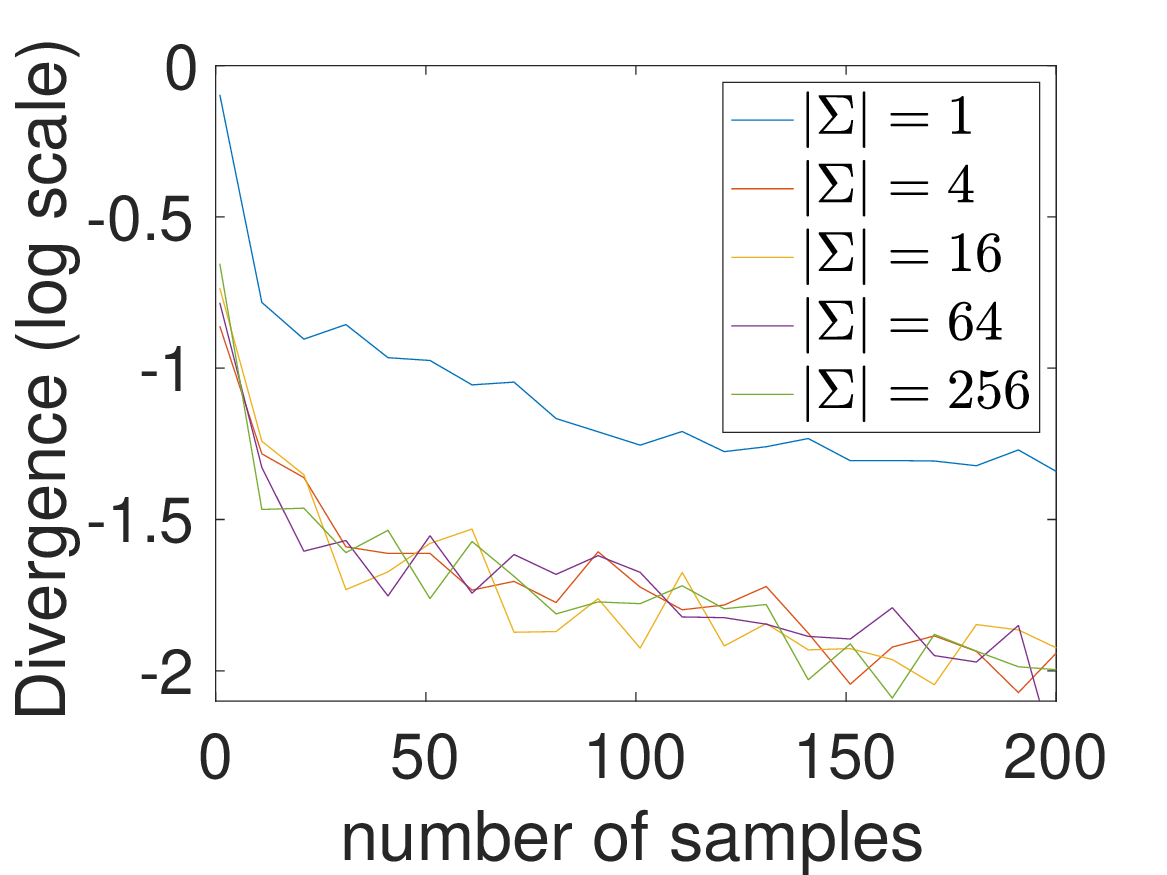}
    \includegraphics[width=.4\columnwidth]{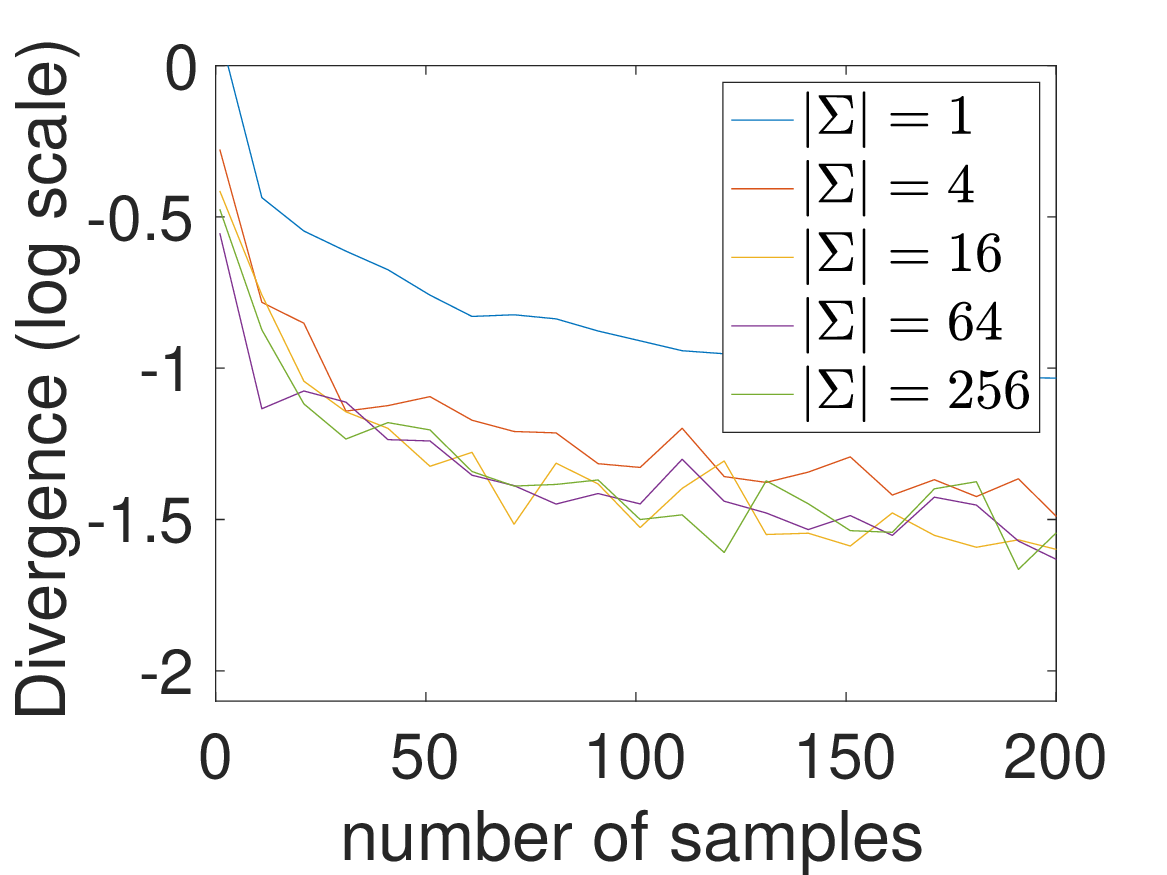}}
\centerline{\includegraphics[width=.4\columnwidth]{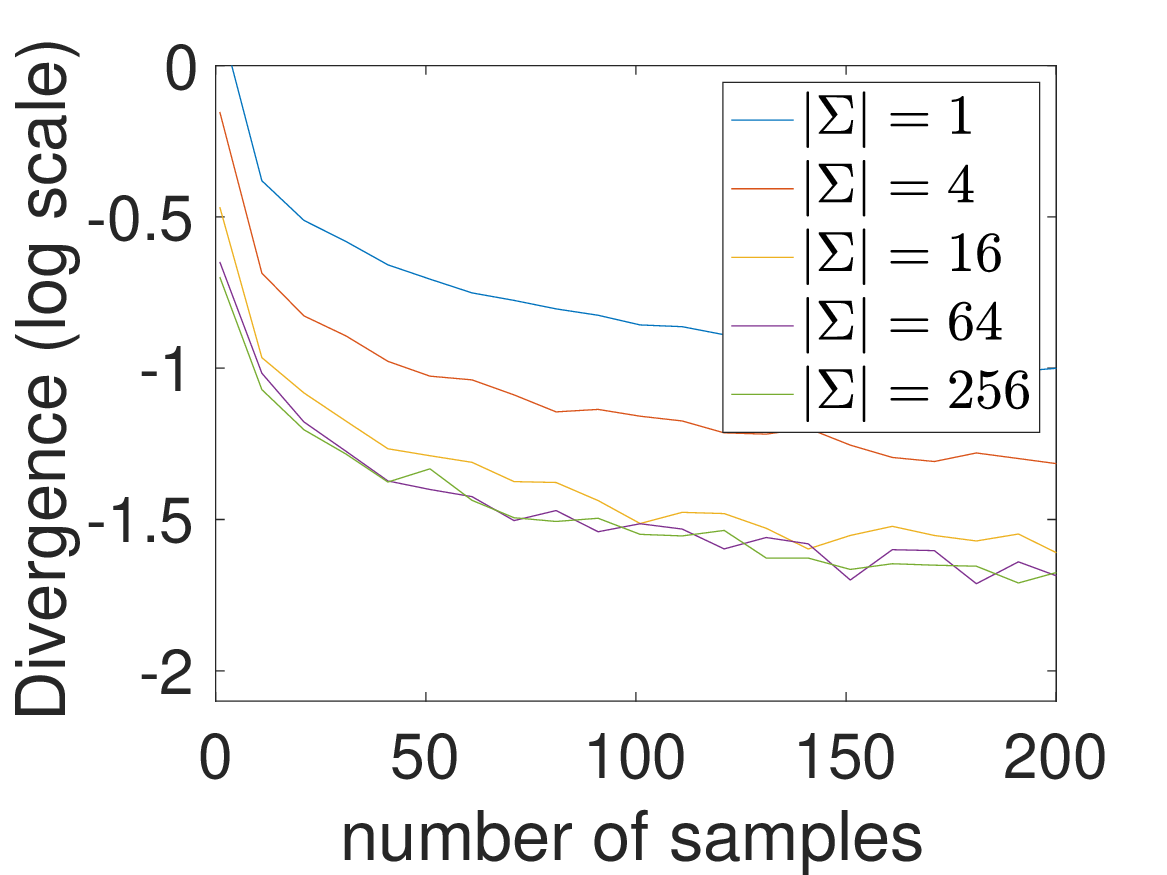}
    \includegraphics[width=.4\columnwidth]{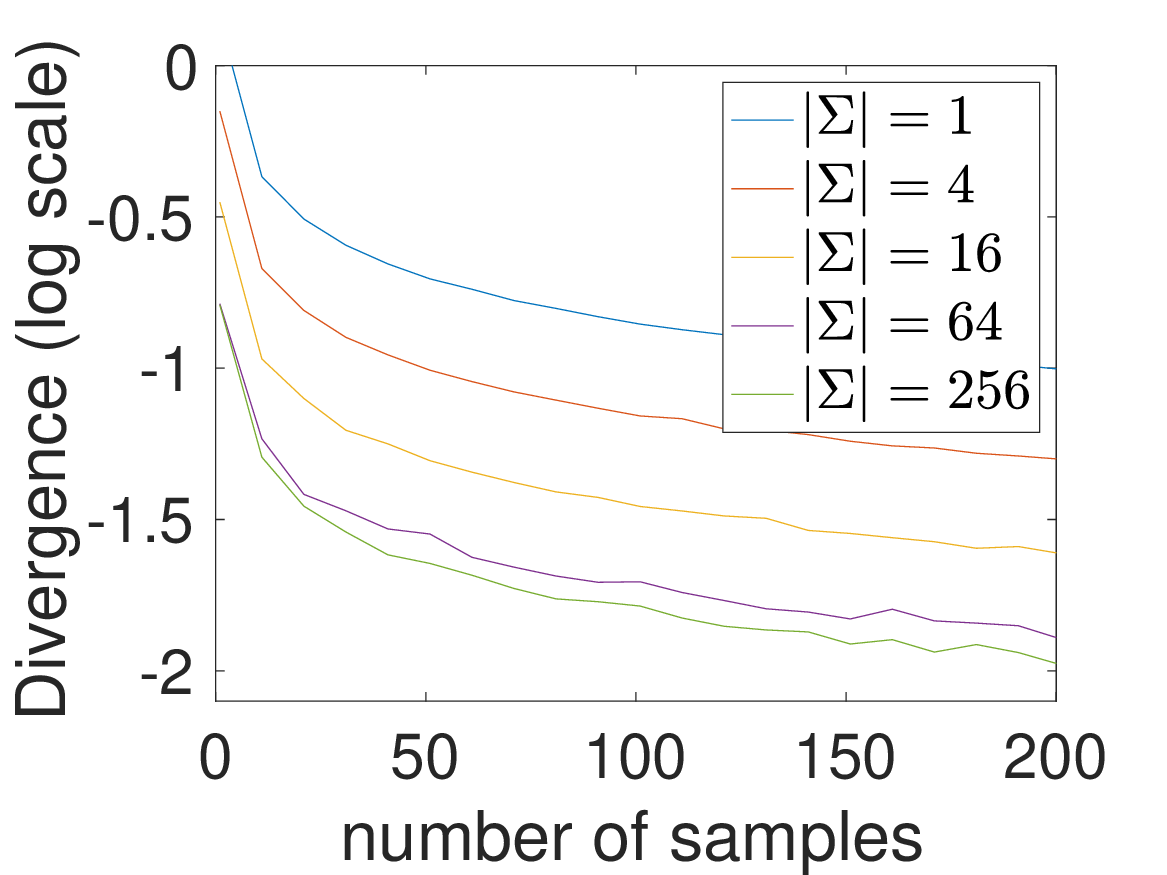}}
\centerline{\includegraphics[width=.4\columnwidth]{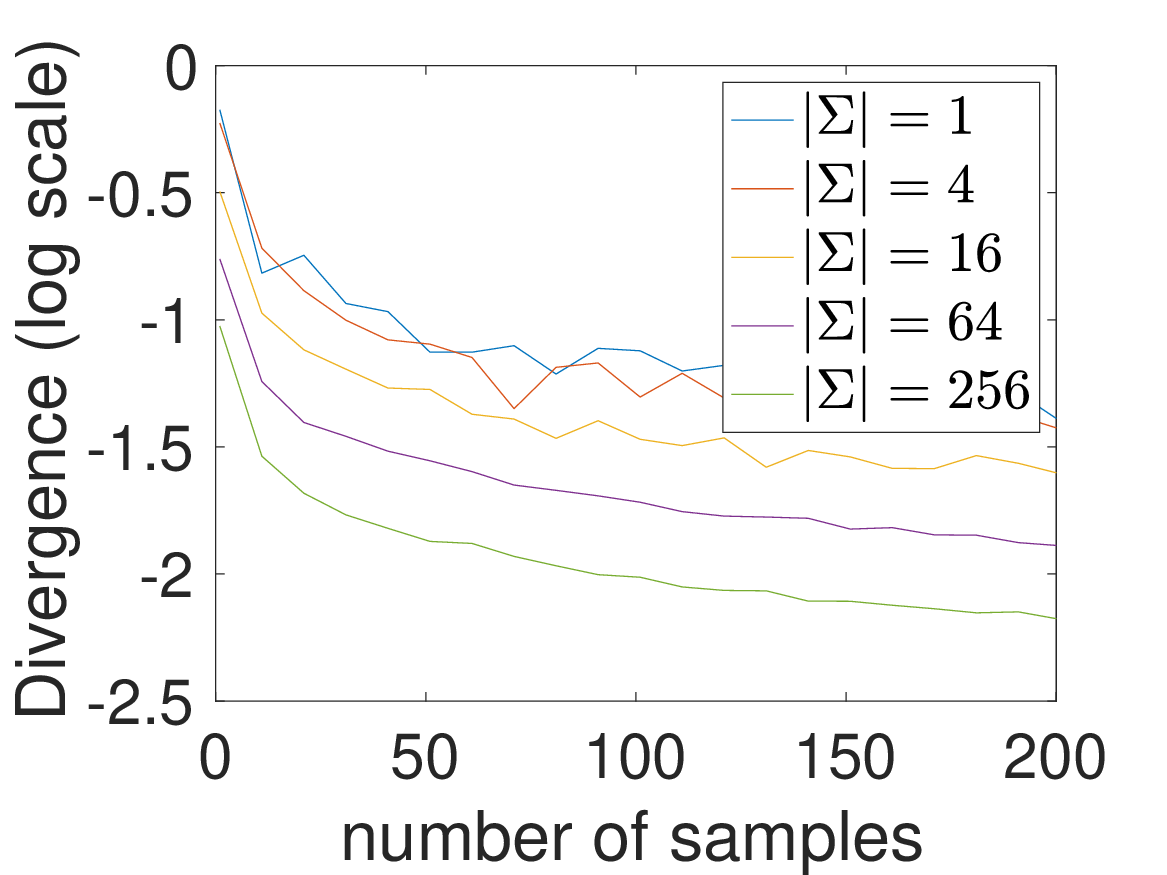}}
    \caption{MMD simulations with Gaussian kernels $k_s(x,y) = e^{-\frac{\norm{x-y}_2^2}{2s^2}}$. From left to right, top to bottom: $s = \frac{2\pi}{1\times 6}$, $s = \frac{2\pi}{4\times6}$, $s = \frac{2\pi}{16\times6}$, $s = \frac{2\pi}{64\times6}$, $s = \frac{2\pi}{6\abs{\Sigma}}$. The first four subfigures (top two rows) show that the Gaussian kernel with a fixed bandwidth $s>0$ satisfies the third condition in Assumption~\ref{assump:kernel} up to a group size of $\abs{\Sigma}=l, l=1, 4, 16, 64$, and thus an improvement of sample complexity of order $C_{\Sigma,k}\approx \abs{\Sigma}^{-1/2}$ persists till $\abs{\Sigma}= l$; when $\abs{\Sigma}>l$, no further reduction in sample complexity can be observed. The last subfigure demonstrates that with an adaptive bandwidth $s$ inversely scaled with $\abs{\Sigma}$, nonstop improvement of the sample complexity can be achieved as the group size $\abs{\Sigma}$ increases. See \cref{example:1} for the detail and explanations. }
	\label{fig:MMD-gaussian}
 \end{center}
 \vskip -0.2in
\end{figure}

\begin{example}\label{example:1}
Let $\mathcal{X} = \{(r\cos\theta,r\sin\theta)\in\mathbb{R}^2: r\in[0,1], \theta\in[0,2\pi)\}$ be the unit disk centered at the origin, and let $k_s(x,y) = e^{-\frac{\norm{x-y}_2^2}{2s^2}}$, $x,y\in\mathcal{X}$, be the Gaussian kernel. Consider the group actions generated by a rotation (with respect to the origin) of $\frac{2 \pi}{l}$, $l = 1, 4, 16, 64, 256$, so that $|\Sigma| = l$ is the group size. The fundamental domain $\CX_0$ under the $\Sigma$-action is $\mathcal{X}_0 = [0,1]\times[0,\frac{2\pi}{l})$ (see \cref{fig:disk} for a visual illustration). We draw samples $x_i\sim P= Q\in \CP_\Sigma(\CX)$ in the following way,
\begin{align*}
    x_i = \sqrt{\xi}_i \left(\cos\left[\frac{2\pi}{l}\theta^{1/3}_i +\eta_i\frac{2\pi}{l}\right], \sin\left[\frac{2\pi}{l}\theta^{1/3}_i +\eta_i\frac{2\pi}{l}\right]\right)
\end{align*}
where $\xi_i$ and $\theta_i$ are i.i.d. uniformly distributed random variables on $[0,1)$ and $\eta_i$ take values over $\{0,\frac{1}{l}, \dots,\frac{l-1}{l}\}$ with equal probabilities. We select the kernel bandwidth $s>0$ in different ways:
\begin{itemize}
    \item Fixed $s$ with changing group size $\abs{\Sigma} = l$. We intuitively follow the ``three-sigma rule'' in the argument direction to pick different $s$. Since the angle of each sector is $\frac{2\pi}{l}$, we select $s = \frac{2\pi}{6l}$, $l = 1, 4, 16, 64$. Smaller bandwidth $s$ corresponds to faster decay of the kernel $k_s(x,y)$, such that for a fixed bandwidth $s = \frac{2\pi}{6l}$, the third condition in Assumption~\ref{assump:kernel} is satisfied with a small $c_k$ for any group $\Sigma$ such that $\abs{\Sigma}\leq l$, i.e., $C_{\Sigma, k} \approx \abs{\Sigma}^{-1/2}$. On the other hand, it is difficult to observe the improvement by further increasing the group size $\abs{\Sigma}$ beyond $\abs{\Sigma}>l$, since the third condition in Assumption~\ref{assump:kernel} is not satisfied with any uniformly small $c$. See the top two rows in \cref{fig:MMD-gaussian} for the results for $s = \frac{2\pi}{l\times 6}$, $l = 1, 4, 16, 64$. Notice that the sample complexity improvement stops right at $|\Sigma| = l$, perfectly matching  our theoretical result \cref{thm:mmd_main}.
    \item $s$ inversely scales with $\abs{\Sigma}$, i.e.,  $s = \frac{2\pi}{6\abs{\Sigma}}$. Unlike the fixed $s$ discusses previously, with these adaptive selections of kernels, we can observed nonstop improvement of the sample complexity as the group size $\abs{\Sigma}$ increases; see the last row of \cref{fig:MMD-gaussian}. This numerical result is explained by the third condition in Assumption~\ref{assump:kernel}; that is, in order to continuously observe the benefit from the increasing group size $\abs{\Sigma}$, we need to have a faster decay in the kernel $k_s$ (i.e., smaller $s$) so that $c_{\Sigma, k_s}$ is uniformly small for all $\abs{\Sigma}$.
\end{itemize}
\end{example}
\begin{remark}
    The bound provided in Theorem~\ref{thm:mmd_main} for the MMD case is almost sharp in the sense that, by a direct calculation, one can obtain that 
    \begin{equation*}
        \frac{E_\CX\text{MMD}^\Sigma(P,P_n)^2}{E_\CX\text{MMD}(P,P_n)^2}\approx C_{\Sigma,k}^2,
    \end{equation*}
    if the Gaussian kernel bandwidth $s\propto \frac{\sqrt{2c_{\Sigma,k}}\pi}{\abs{\Sigma}}$.
\end{remark}

\section{Conclusion and future work}\label{sec:conclusion}
We provide rigorous analysis to quantify the reduction in sample complexity for variational divergence estimations between group-invariant distributions. We obtain a reduction in the error bound by a power of the group size when the group is finite. The exponent on the group size depends on the intrinsic dimension of the support, characterized by the covering number rate for the Wasserstein-1 metric and the Lipschitz-regularized $\alpha$-divergence; that reduction, however, is independent of the ambient dimension for the MMD as in \cite{gretton2006kernel,gretton2007kernel}.

This work also motivates some possible future directions. For the Wasserstein-1 metric in $\mathbb{R}^2$, one could potentially derive a sharper bound in terms of the group size. For the MMD with Gaussian kernels, it is worth investigating how to choose the bandwidth to make as much use of the group structure as possible. Further applications of the theories on machine learning, such as neural generative models or neural estimations of divergence under symmetry, are also expected.

\section{Proofs}\label{appendix:proofs}
In this section, we provide detailed statements of the theorems introduced in \cref{sec:sample_complexity} as well as their proofs.
\subsection{Wasserstein-1 metric}\label{appendix:gammaIPM}
\begin{assumption}\label{assumption:sapiro}
    Let $\mathcal{X} = \Sigma\times \mathcal{X}_0\subset\mathbb{R}^d$. Assume that there exists some $\delta_0>0$ such that\\
    \begin{enumerate}[label={\arabic*)}]
        \item $\norm{\sigma(x)-\sigma'(x')}_2>2\delta_0$, $\forall x,x'\in\mathcal{X}_0, \sigma\neq\sigma'\in\Sigma$; and
        \item $\norm{\sigma(x)-\sigma(x')}_2\geq\norm{x-x'}_2$, $\forall x,x'\in\mathcal{X}_0, \sigma\in\Sigma$,
    \end{enumerate}
    where $\|\cdot\|_2$ is the Euclidean norm on $\R^d$.
\end{assumption}
\cref{example:wss1d} provides a simple example when this assumption holds.
\begin{theorem}\label{thm:gammaIPM}
Let $\mathcal{X} = \Sigma\times\mathcal{X}_0$ be a subset of $\mathbb{R}^D$ satisfying the conditions in Assumption~\ref{assumption:sapiro}. Assume that $\mathcal{N}(\mathcal{X},\delta)\lesssim \delta^{-d}$ for sufficiently small $\delta$. Suppose $P$ and $Q$ are $\Sigma$-invariant probability measures on $\mathcal{X}$. 

1) If $d\geq 2$, then for any $s>0, \epsilon>0$ and $m,n$ sufficiently large, we have with probability at least $1-\epsilon$,
\begin{align*}
\abs{W(P,Q)-W^{\Sigma}(P_m,Q_n)}&\leq\left(8+\frac{24}{(\frac{d+s}{2}-1)}\right)\left[\left(\frac{9D_{\mathcal{X},L}^2}{\abs{\Sigma}m}\right)^{\frac{1}{d+s}}+\left(\frac{9D_{\mathcal{X},L}^2}{\abs{\Sigma}n}\right)^{\frac{1}{d+s}}\right]\\
&\quad+ \bar{D}_{\mathcal{X}_0, L}\left(\frac{24}{\sqrt{m}}+ \frac{24}{\sqrt{n}}\right)+L\cdot \text{\normalfont diam}(\mathcal{X}_0)\sqrt{\frac{2(m+n)}{mn}\ln \frac{1}{\epsilon}},
\end{align*}
where $D_{\mathcal{X}, L}$ depends only on $\mathcal{X}$ and $L$; $\bar{D}_{\mathcal{X}_0, L}$ depends only on $\mathcal{X}_0$ and $L$, and is increasing in $\mathcal{X}_0$, i.e., $\bar{D}_{A_1, L}\le \bar{D}_{A_2, L}$ for $A_1\subset A_2$;\\
2) If $d=1$, then for any $\epsilon>0$ and $m,n$ sufficiently large, we have with probability at least $1-\epsilon$,
\begin{align*}
\abs{W(P,Q)-W^{\Sigma}(P_m,Q_n)}&\leq cL\cdot\text{\normalfont diam}(\mathcal{X}_0) \left(\frac{1}{\sqrt{m}} + \frac{1}{\sqrt{n}}\right) +L\cdot \text{\normalfont diam}(\mathcal{X}_0)\sqrt{\frac{2(m+n)}{mn}\ln \frac{1}{\epsilon}},
\end{align*}
where $c>0$ is an absolute constant independent of $\mathcal{X}$ and $\mathcal{X}_0$.
\end{theorem}
Before proving this theorem, we have the following lemmas.

\begin{lemma}\label{lemma:symmetrizeGamma}
Suppose the $\Sigma$-actions on $\mathcal{X}$ are $1$-Lipschitz, i.e., $\norm{\sigma x-\sigma y}_2\leq\norm{x-y}_2$ for any $x,y\in\mathcal{X}$ and $\sigma\in\Sigma$, then we have $S_\Sigma[\Gamma]\subset\Gamma$, where $\Gamma = \text{Lip}_L(\mathcal{X})$.
\end{lemma}
\begin{proof}
For any $x,y\in\mathcal{X}$ and $f\in\Gamma$, we have
\begin{align*}
    \abs{S_\Sigma(f)(x)-S_\Sigma(f)(y)}
    &= \abs{\frac{1}{\abs{\Sigma}}\sum_{\sigma\in\Sigma}f\left(\sigma x\right) - \frac{1}{\abs{\Sigma}}\sum_{\sigma\in\Sigma}f\left(\sigma y\right)}\\
    &\leq \frac{1}{\abs{\Sigma}}\sum_{\sigma\in\Sigma}\abs{f\left(\sigma x\right) - f\left(\sigma y\right)}\\
    &\leq \frac{1}{\abs{\Sigma}}\sum_{\sigma\in\Sigma}L\norm{\sigma x - \sigma y}_2\\
    &\leq \frac{1}{\abs{\Sigma}}\sum_{\sigma\in\Sigma}L\norm{x - y}_2\\
    &=L\norm{x - y}_2.
\end{align*}
Therefore, we have $S_\Sigma(f)\in\Gamma$.
\end{proof}

\begin{lemma}\label{lemma:infinitybound}
For any $\gamma\in\text{Lip}_L(\mathcal{X}_0)$, there exists $\nu\in\mathbb{R}$, such that $\norm{\gamma+\nu}_\infty\leq L\cdot\text{\normalfont diam}(\mathcal{X}_0)$. 
\end{lemma}
\begin{proof}
Suppose $\gamma\in\text{Lip}_L(\mathcal{X}_0)$ and $\norm{\gamma(x)}_\infty > L\cdot\text{diam}(\mathcal{X}_0)$. Without loss of generality, we can assume $\sup_{x\in\mathcal{X}_0}\gamma(x)>L\cdot\text{diam}(\mathcal{X}_0)$. Since $\gamma$ is $L$-Lipschitz on $\mathcal{X}_0$, we have $\sup_{x\in\mathcal{X}_0}\gamma(x)- \inf_{x\in\mathcal{X}_0}\gamma(x)\leq L\cdot\text{diam}(\mathcal{X}_0)$, so that 
\[
\inf_{x\in\mathcal{X}_0}\gamma(x)\geq \sup_{x\in\mathcal{X}_0}\gamma(x)-L\cdot\text{diam}(\mathcal{X}_0)>0.
\]
Hence we can select $\nu=-\frac{\inf_{x\in\mathcal{X}_0}\gamma(x)}{2}$, so that $\norm{\gamma+\nu}_\infty<\norm{\gamma}_\infty$.
\end{proof}

We provide a variant of the Dudley's entropy integral as well as its proof for completeness.
\begin{lemma}\label{lemma:Dudley}
    Suppose $\mathcal{F}$ is a family of functions mapping the metric space $(\mathcal{X},\rho)$ to $[-M,M]$ for some $M>0$. Also assume that $0\in\mathcal{F}$ and $\mathcal{F} = -\mathcal{F}$. Let $\xi=\{\xi_1,\dots,\xi_m\}$ be a set of independent random variables that take values on $\{-1,1\}$ with equal probabilities, $i = 1,\dots,m$. $x_1,x_2,\dots,x_m\in\mathcal{X}$. Then we have 
    \begin{equation*}
    E_{\xi}\sup_{f\in\mathcal{F}}\abs{\frac{1}{m}\sum_{i=1}^m\xi_if(x_i)}\leq\inf_{\alpha>0} 4\alpha+\frac{12}{\sqrt{m}}\int_{\alpha}^{M}\sqrt{\ln\mathcal{N(\mathcal{F},\delta,\norm{\cdot}_{\infty})}}\diff{\delta}.
    \end{equation*}
\end{lemma}
The proof of \cref{lemma:Dudley} is standard using the dyadic path., e.g. see the proof of Lemma A.5. in \cite{bartlett2017spectrally}.
\begin{proof}\label{proof:Dudley}

    Let $N$ be an arbitrary positive integer and $\delta_k = M2^{-(k-1)}$, $k=1,\dots,N$. Let $V_k$ be the cover achieving $\mathcal{N}(\mathcal{F},\delta_k,\norm{\cdot}_\infty)$ and denote $\abs{V_k} = \mathcal{N}(\mathcal{F},\delta_k,\norm{\cdot}_\infty)$. For any $f\in\mathcal{F}$, let $\pi_k(f)\in V_k$, such that $\norm{f-\pi_k(f)}_\infty\leq\delta_k$. We have
    \begin{align*}      &E_\xi\sup_{f\in\mathcal{F}}\abs{\frac{1}{m}\sum_{i=1}^m\xi_i f(x_i)}\\
    &\leq E_\xi\sup_{f\in\mathcal{F}}\abs{\frac{1}{m}\sum_{i=1}^m\xi_i\left( f(x_i)-\pi_N(f)(x_i)\right)} + \sum_{j=1}^{N-1}E_\xi\sup_{f\in\mathcal{F}}\abs{\frac{1}{m}\sum_{i=1}^m\xi_i\left( \pi_{j+1}(f)(x_i)-\pi_j(f)(x_i)\right)}\\
    &\quad + E_\xi\sup_{f\in\mathcal{F}}\abs{\frac{1}{m}\sum_{i=1}^m\xi_i \pi_1(f)(x_i)}.
    \end{align*}
    The first on the right hand side is bounded by $\delta_N$. Note that we can choose $V_1=\{0\}$, so that $\pi_1(f)$ is the zero function. For each $j$, let $W_j = \{\pi_{j+1}(f)-\pi_j(f):f\in\mathcal{F}\}$. We have $\abs{W_j}\leq\abs{V_{j+1}}\abs{V_j}\leq\abs{V_{j+1}}^2$. Then we have
    \[
    \sum_{j=1}^{N-1}E_\xi\sup_{f\in\mathcal{F}}\abs{\frac{1}{m}\sum_{i=1}^m\xi_i\left( \pi_{j+1}(f)(x_i)-\pi_j(f)(x_i)\right)} = \sum_{j=1}^{N-1}E_\xi\sup_{w\in W_j}\abs{\frac{1}{m}\sum_{i=1}^m\xi_i w(x_i)}.
    \]
    In addition, we have
    \begin{align*}
        &\sup_{w\in W_j}\sqrt{\sum_{i=1}^m w(x_i)^2}\\
        &= \sup_{f\in\mathcal{F}}\sqrt{\sum_{i=1}^m\left(\pi_{j+1}(f)(x_i)-\pi_j(f)(x_i)\right)^2}\\
        &\leq \sup_{f\in\mathcal{F}}\sqrt{\sum_{i=1}^m\left(\pi_{j+1}(f)(x_i)-f(x_i)\right)^2} + \sup_{f\in\mathcal{F}}\sqrt{\sum_{i=1}^m\left(f(x_i)-\pi_j(f)(x_i)\right)^2}\\
        &\leq \sqrt{m\cdot\delta_{j+1}^2} + \sqrt{m\cdot\delta_{j}^2}\\
        &=\sqrt{m}(\delta_{j+1}+\delta_j)\\
        &=3\sqrt{m}\delta_{j+1}.
    \end{align*}
By the Massart finite class lemma (see, e.g. \cite{mohri2018foundations}), we have
\begin{align*}
    E_\xi\sup_{w\in W_j}\abs{\frac{1}{m}\sum_{i=1}^m\xi_i w(x_i)}\leq\frac{3\sqrt{m}\delta_{j+1}\sqrt{2\ln\abs{W_j}}}{m}\leq\frac{6\delta_{j+1}\sqrt{\ln\abs{V_{j+1}}}}{\sqrt{m}}.
\end{align*}
Therefore,
\begin{align*}
    E_\xi\sup_{f\in\mathcal{F}}\abs{\frac{1}{m}\sum_{i=1}^m\xi_i f(x_i)} &\leq \delta_N +\frac{6}{\sqrt{m}}\sum_{j=1}^{N-1}\delta_{j+1}\sqrt{\ln\mathcal{N}(\mathcal{F},\delta_{j+1},\norm{\cdot}_{\infty})}\\
    &\leq \delta_N +\frac{12}{\sqrt{m}}\sum_{j=1}^{N}(\delta_j-\delta_{j+1})\sqrt{\ln\mathcal{N}(\mathcal{F},\delta_{j},\norm{\cdot}_{\infty})}\\
    &\leq\delta_N + \frac{12}{\sqrt{m}}\int_{\delta_{N+1}}^M\sqrt{\ln\mathcal{N}(\mathcal{F},\delta,\norm{\cdot}_{\infty})}\diff{\delta}.
\end{align*}
Finally, select any $\alpha\in(0,M)$ and let $N$ be the largest integer with $\delta_{N+1}>\alpha$, (implying $\delta_{N+2}\leq \alpha$ and $\delta_N=4\delta_{N+2}\leq 4\alpha$), so that
\[
\delta_N + \frac{12}{\sqrt{m}}\int_{\delta_{N+1}}^M\sqrt{\ln\mathcal{N}(\mathcal{F},\delta_{j},\norm{\cdot}_{\infty})}\diff{\delta}\leq 4\alpha + \frac{12}{\sqrt{m}}\int_{\alpha}^M\sqrt{\ln\mathcal{N}(\mathcal{F},\delta,\norm{\cdot}_{\infty})}\diff{\delta}.
\]
\end{proof}

We can easily extend Lemma 6 in \cite{gottlieb2017efficient} to the following lemma by meshing on the range $[-M,M]$ rather than $[0,1]$.
\begin{lemma}\label{lemma:coverbycover}
    Let $\mathcal{F}$ be the family of $L$-Lipschitz functions mapping the metric space $(\mathcal{X},\norm{\cdot}_2)$ to $[-M,M]$ for some $M>0$. Then we have 
    \begin{equation*}
        \mathcal{N}(\mathcal{F},\delta,\norm{\cdot}_{\infty})\leq(\frac{c_1M}{\delta})^{\mathcal{N}(\mathcal{X},\frac{c_2\delta}{L})},
    \end{equation*}
    where $c_1\geq1$ and $c_2\leq1$ are some absolute constants not depending on $\mathcal{X}$, $M$, and $\delta$.
\end{lemma}
\begin{remark}\label{rmk:coverbycover}
    If $\CX$ is connected, then the bound can be improved to $\mathcal{N}(\mathcal{F},\delta,\norm{\cdot}_{\infty})\leq e^{\mathcal{N}(\mathcal{X},\frac{c_2\delta}{L})}$ by the result in \cite{kolmogorov1961varepsilon}.
\end{remark}

\begin{lemma}[Theorem 3 in \cite{sokolic2017generalization}]\label{lemma:ratioofcovering}
Assume that $\mathcal{X}=\Sigma\times\mathcal{X}_0$. If for some $\delta>0$ we have\\
1) $\norm{\sigma(x)-\sigma'(x')}_2>2\delta$, $\forall x,x'\in\mathcal{X}_0, \sigma\neq\sigma'\in\Sigma$; and\\
2) $\norm{\sigma(x)-\sigma(x')}_2\geq\norm{x-x'}_2$, $\forall x,x'\in\mathcal{X}_0, \sigma\in\Sigma$,\\
then we have
\[
\frac{\mathcal{N}(\mathcal{X}_0,\delta)}{\mathcal{N}(\mathcal{X},\delta)}\leq\frac{1}{\abs{\Sigma}}.
\]
\end{lemma}

In addition, we provide the following lemma for the scaling of covering numbers.
\begin{lemma}\label{lemma:scalingofcoveringnumber}
    Let $\mathcal{X}$ be a subset of $\mathbb{R}^d$ or $\mathcal{N}(\mathcal{X},\delta)\lesssim \delta^{-d}$, and $\bar{\delta}>0$. Then there exists a constant $C_{d,\bar{\delta}}$ that depends on $d$ and $\bar{\delta}$ such that for $\delta\in(0,1)$ we have
    \begin{align*}
        \mathcal{N}(\mathcal{X},\delta)\leq C_{d,\bar{\delta}}\cdot\frac{\mathcal{N}(\mathcal{X},\bar{\delta})}{\delta^d}.
    \end{align*}
\end{lemma}
\begin{proof}
    Let $N:=\mathcal{N}(\mathcal{X},\bar{\delta})$. Then $\mathcal{X}$ can be covered by $N$ balls with radius $\bar{\delta}$. From Proposition 4.2.12 in \cite{vershynin2018high}, we know that each ball with radius $\bar{\delta}$ can be covered by $\frac{(\bar{\delta}+\delta/2)^d}{(\delta/2)^d}$ balls with radius $\delta$. This implies that $\mathcal{X}$ can be covered by $N\cdot\frac{(\bar{\delta}+\delta/2)^d}{(\delta/2)^d}$ balls with radius $\delta$, so that $\mathcal{N}^{\text{ext}}(\mathcal{X},\delta)\leq N\cdot\frac{(\bar{\delta}+\delta/2)^d}{(\delta/2)^d}$, where $\mathcal{N}^{\text{ext}}(\mathcal{X},\delta)$ is the exterior covering number of $\mathcal{X}$ with radius $\delta$. Therefore, $\mathcal{N}(\mathcal{X},\delta)\leq\mathcal{N}^{\text{ext}}(\mathcal{X},\delta/2)\leq N\cdot\frac{(\bar{\delta}+\delta/4)^d}{(\delta/4)^d} = N\cdot(\frac{4\bar{\delta}}{\delta}+1)^d\leq N\cdot\frac{(4\bar{\delta}+1)^d}{\delta^d}$.
\end{proof}

\begin{proof}[Proof of \cref{thm:gammaIPM}]
\begin{align}\label{eq:IPMsup}
    &\abs{W(P,Q)-W^{\Sigma}(P_m,Q_n)}\nonumber\\
    &= \abs{\sup_{\gamma\in\Gamma_{\Sigma}}\left\{E_P[\gamma]-E_Q[\gamma]\right\} - \sup_{\gamma\in\Gamma_{\Sigma}}\left\{E_{P_m}[\gamma]-E_{Q_n}[\gamma]\right\}}\nonumber\\
    &\leq\sup_{\gamma\in\Gamma_{\Sigma}}\abs{E_P[\gamma]-\frac{1}{m}\sum_{i=1}^m\gamma(x_i)-\left(E_Q[\gamma]-\frac{1}{n}\sum_{i=1}^n\gamma(y_i)\right)}\nonumber\\
    &=\sup_{\gamma\in\Gamma_{\Sigma}}\abs{E_P[\gamma]-\frac{1}{m}\sum_{i=1}^m\gamma\left(T_0(x_i)\right)-\left(E_Q[\gamma]-\frac{1}{n}\sum_{i=1}^n\gamma\left(T_0(y_i)\right)\right)}\nonumber\\
    &\stackrel{\text{$(a)$}}{\leq} \sup_{\gamma\in\text{Lip}_L(\mathcal{X}_0)}\abs{E_{P_{\mathcal{X}_0}}[\gamma]-\frac{1}{m}\sum_{i=1}^m\gamma\left(T_0(x_i)\right)-\left(E_{Q_{\mathcal{X}_0}}[\gamma]-\frac{1}{n}\sum_{i=1}^n\gamma\left(T_0(y_i)\right)\right)}\\
    &:=f(x_1,\dots,x_m,y_1,\dots,y_n)\nonumber,
\end{align}
where inequality $(a)$ is due to the fact that $E_P[\gamma] = E_{P_{\mathcal{X}_0}}[\gamma|_{\mathcal{X}_0}]$ and $E_Q[\gamma] = E_{Q_{\mathcal{X}_0}}[\gamma|_{\mathcal{X}_0}]$ since $P$ and $Q$ are both $\Sigma$-invariant and $\gamma\in\Gamma_{\Sigma}$, and the fact that if $\gamma\in\Gamma_{\Sigma}$, then $\gamma|_{\mathcal{X}_0}\in\text{Lip}_L(\mathcal{X}_0)$, where $\gamma|_{\mathcal{X}_0}$ is the restriction of $\gamma$ on $\mathcal{X}_0$. 

Note that the quantity inside the absolute value in \eqref{eq:IPMsup} will not change if we replace $\gamma$ by $\gamma+\nu$ and we still have $\gamma+\nu\in\text{Lip}_L(\mathcal{X}_0)$ for any $\nu\in\mathbb{R}$. Therefore, by Lemma \ref{lemma:infinitybound}, the supremum in \eqref{eq:IPMsup} can be taken over $\gamma\in\text{Lip}_L(\mathcal{X}_0)$, where $\norm{\gamma}_{\infty}\leq L\cdot\text{diam}(\mathcal{X}_0)$. The denominator in the exponent when applying the McDiarmid’s inequality is thus equal to
\begin{equation}\label{eq:mcdiarmid}
m\left(\frac{2L\cdot\text{diam}(\mathcal{X}_0)}{m}\right)^2+n\left(\frac{2L\cdot\text{diam}(\mathcal{X}_0)}{n}\right)^2 = 4L^2\cdot\text{diam}(\mathcal{X}_0)^2\frac{m+n}{mn}.
\end{equation}
Denoting by $X' = \{x'_1,x'_2,\dots,x'_m\}$ and $Y' = \{y'_1,y'_2,\dots,y'_n\}$ the i.i.d. samples drawn from $P_{\mathcal{X}_0}$ and $Q_{\mathcal{X}_0}$. Also note that $T_0(x_1),\dots,T_0(x_m)$ and $T_0(y_1),\dots,T_0(y_n)$ can be viewed as i.i.d. samples on $\mathcal{X}_0$ drawn from $P_{\mathcal{X}_0}$ and $Q_{\mathcal{X}_0}$ respectively, such that the expectation
\begin{align*}
&E_{X,Y}f(x_1,x_2,\dots,x_m,y_1,y_2,\dots,y_n)\\
&= E_{X,Y}\sup_{\gamma\in\text{Lip}_L(\mathcal{X}_0)}\abs{E_{P_{\mathcal{X}_0}}[\gamma]-\frac{1}{m}\sum_{i=1}^m\gamma(T_0(x_i))-\left(E_{Q_{\mathcal{X}_0}}[\gamma]-\frac{1}{n}\sum_{i=1}^n\gamma(T_0(y_i))\right)}
\end{align*}
can be replaced by the equivalent quantity
\[
E_{X,Y}\sup_{\gamma\in\text{Lip}_L(\mathcal{X}_0)}\abs{E_{P_{\mathcal{X}_0}}[\gamma]-\frac{1}{m}\sum_{i=1}^m\gamma(x_i)-\left(E_{Q_{\mathcal{X}_0}}[\gamma]-\frac{1}{n}\sum_{i=1}^n\gamma(y_i)\right)},
\]
where $X = \{x_1,x_2,\dots,x_m\}$ and $Y =\{y_1,y_2,\dots,y_n\}$ are are i.i.d. samples on $\mathcal{X}_0$ drawn from $P_{\mathcal{X}_0}$ and $Q_{\mathcal{X}_0}$ respectively. Then we have
\begin{align*}
&E_{X,Y}\sup_{\gamma\in\text{Lip}_L(\mathcal{X}_0)}\abs{E_{P_{\mathcal{X}_0}}[\gamma]-\frac{1}{m}\sum_{i=1}^m\gamma(x_i)-\left(E_{Q_{\mathcal{X}_0}}[\gamma]-\frac{1}{n}\sum_{i=1}^n\gamma(y_i)\right)}\\
&= E_{X,Y}\sup_{\gamma\in\text{Lip}_L(\mathcal{X}_0)}\abs{E_{X'}\left(\frac{1}{m}\sum_{i=1}^m\gamma(x'_i)\right)-\frac{1}{m}\sum_{i=1}^m\gamma(x_i)-E_{Y'}\left(\frac{1}{n}\sum_{i=1}^n\gamma(y'_i)\right)+\frac{1}{n}\sum_{i=1}^n\gamma(y_i)}\\
&\leq E_{X,Y,X',Y'}\sup_{\gamma\in\text{Lip}_L(\mathcal{X}_0)}\abs{\frac{1}{m}\sum_{i=1}^m\gamma(x'_i)-\frac{1}{m}\sum_{i=1}^m\gamma(x_i)-\frac{1}{n}\sum_{i=1}^n\gamma(y'_i)+\frac{1}{n}\sum_{i=1}^n\gamma(y_i)}\\
&=E_{X,Y,X',Y',\xi,\xi'}\sup_{\gamma\in\text{Lip}_L(\mathcal{X}_0)}\abs{\frac{1}{m}\sum_{i=1}^m\xi_i\left(\gamma(x'_i)-\gamma(x_i)\right)-\frac{1}{n}\sum_{i=1}^n\xi'_i\left(\gamma(y'_i)-\gamma(y_i)\right)}\\
&\leq E_{X,X',\xi}\sup_{\gamma\in\text{Lip}_L(\mathcal{X}_0)}\abs{\frac{1}{m}\sum_{i=1}^m\xi_i\left(\gamma(x'_i)-\gamma(x_i)\right)} + E_{Y,Y',\xi'}\sup_{\gamma\in\text{Lip}_L(\mathcal{X}_0)}\abs{\frac{1}{n}\sum_{i=1}^n\xi'_i\left(\gamma(y'_i)-\gamma(y_i)\right)}\\
&\leq \inf_{\alpha>0} 8\alpha+\frac{24}{\sqrt{m}}\int_{\alpha}^{M}\sqrt{\ln\mathcal{N}(\mathcal{F}_0,\delta,\norm{\cdot}_{\infty})}\diff{\delta} + \inf_{\alpha>0} 8\alpha+\frac{24}{\sqrt{n}}\int_{\alpha}^{M}\sqrt{\ln\mathcal{N}(\mathcal{F}_0,\delta,\norm{\cdot}_{\infty})}\diff{\delta},
\end{align*}
where $\mathcal{F}_0 = \{\gamma\in\text{Lip}_L(\mathcal{X}_0):\norm{\gamma}_\infty\leq M\}$ and $M = L\cdot\text{diam}(\mathcal{X}_0)$ by Lemma \ref{lemma:infinitybound}. 

For $d\geq2$, from Lemma \ref{lemma:coverbycover}, we have $\ln\mathcal{N}(\mathcal{F}_0,\delta,\norm{\cdot}_{\infty})\leq\mathcal{N}(\mathcal{X}_0,\frac{c_2\delta}{L})\ln(\frac{c_1M}{\delta})$. We fix a $\bar{\delta}>0$ such that $\mathcal{N}(\mathcal{X},\frac{c_2\bar{\delta}}{L})=1$, and select $\delta^*$ such that $\frac{c_2\delta^*}{L}\leq1$ and $\frac{c_2\delta^*}{L}\leq\delta_0$; that is, $\delta^*\leq\min\left(\frac{L}{c_2},\frac{L\delta_0}{c_2}\right)$, so that by Lemma \ref{lemma:ratioofcovering} and \ref{lemma:scalingofcoveringnumber}, we have 
\[
\mathcal{N}(\mathcal{X}_0,\frac{c_2\delta}{L})\ln(\frac{c_1M}{\delta})\leq\frac{\mathcal{N}(\mathcal{X},\frac{c_2\delta}{L})}{\abs{\Sigma}}\ln(\frac{c_1M}{\delta})\leq\frac{C_{d,\bar{\delta}}L^d}{\abs{\Sigma}c_2^d\delta^d}\ln(\frac{c_1M}{\delta}),
\]
when $\delta<\delta^*$. Therefore, for sufficiently small $\alpha$, we have 
\begin{align}\label{eq:dudleybound}
&\int_{\alpha}^{M}\sqrt{\ln\mathcal{N}(\mathcal{F}_0,\delta,\norm{\cdot}_{\infty})}\diff{\delta}\nonumber\\
&= \int_{\alpha}^{\delta^*}\sqrt{\ln\mathcal{N}(\mathcal{F}_0,\delta,\norm{\cdot}_{\infty})}\diff{\delta} + \int_{\delta^*}^{M}\sqrt{\ln\mathcal{N}(\mathcal{F}_0,\delta,\norm{\cdot}_{\infty})}\diff{\delta}\nonumber\\
&\leq \int_{\alpha}^{\delta^*}\sqrt{\frac{C_{d,\bar{\delta}}L^d}{\abs{\Sigma}c_2^d\delta^d}\ln(\frac{c_1M}{\delta})}\diff{\delta} + \int_{\delta^*}^{M}\sqrt{\ln\mathcal{N}(\mathcal{F}_0,\delta,\norm{\cdot}_{\infty})}\diff{\delta}.
\end{align}
For any $s>0$, we can choose $\delta^*$ to be sufficiently small, such that we have $\ln(\frac{c_1M}{\delta})\leq\frac{1}{\delta^s}$ when $\delta\leq\delta^*$. Therefore, if we let $D_{\mathcal{X},L}=\sqrt{\frac{C_{d,\bar{\delta}}L^d}{c_2^d}}$, we will have 
\begin{align*}      \int_{\alpha}^{\delta^*}\sqrt{\frac{C_{d,\bar{\delta}}L^d}{\abs{\Sigma}c_2^d\delta^d}\ln(\frac{c_1M}{\delta})}\diff{\delta}
&\leq D_{\mathcal{X},L}\int_{\alpha}^{\delta^*}\sqrt{\frac{1}{\abs{\Sigma}\delta^{d+s}}}\diff{\delta}\\
&\leq D_{\mathcal{X},L}\int_{\alpha}^{\infty}\sqrt{\frac{1}{\abs{\Sigma}\delta^{d+s}}}\diff{\delta}\\
&=\frac{D_{\mathcal{X},L}}{\sqrt{\abs{\Sigma}}}\cdot\frac{\alpha^{1-\frac{d+s}{2}}}{\frac{d+s}{2}-1}.
\end{align*}
Notice that the second integral in \eqref{eq:dudleybound} is bounded while the first integral diverges as $\alpha$ tends to zero, so we can optimize the majorizing terms
\[
8\alpha + \frac{24}{\sqrt{m}}\cdot\frac{D_{\mathcal{X},L}}{\sqrt{\abs{\Sigma}}}\cdot\frac{\alpha^{1-\frac{d+s}{2}}}{\frac{d+s}{2}-1}
\]
with respect to $\alpha$, to obtain \[
\alpha = (\frac{9}{m})^{\frac{1}{d+s}}\cdot(\frac{D_{\mathcal{X},L}^2}{\abs{\Sigma}})^{\frac{1}{d+s}},
\]
so that
\begin{align*}
    &\inf_{\alpha>0} 8\alpha+\frac{24}{\sqrt{m}}\int_{\alpha}^{M}\sqrt{\ln\mathcal{N}(\mathcal{F}_0,\delta,\norm{\cdot}_{\infty})}\diff{\delta}\\
    &\leq 8(\frac{9}{m})^{\frac{1}{d+s}}\cdot(\frac{D_{\mathcal{X},L}^2}{\abs{\Sigma}})^{\frac{1}{d+s}} + \frac{24}{(\frac{d+s}{2}-1)}(\frac{9}{m})^{\frac{1}{d+s}}\cdot(\frac{D_{\mathcal{X},L}^2}{\abs{\Sigma}})^{\frac{1}{d+s}} + \frac{24}{\sqrt{m}}\int_{\delta^*}^{M}\sqrt{\ln\mathcal{N}(\mathcal{F}_0,\delta,\norm{\cdot}_{\infty})}\diff{\delta}.
\end{align*}
Therefore, for sufficiently large $m$ and $n$, we have 
\begin{align*}
&E_{X,Y}\sup_{\gamma\in\text{Lip}_L(\mathcal{X}_0)}\abs{E_{P_{\mathcal{X}_0}}[\gamma]-\frac{1}{m}\sum_{i=1}^m\gamma(x_i)-\left(E_{Q_{\mathcal{X}_0}}[\gamma]-\frac{1}{n}\sum_{i=1}^n\gamma(y_i)\right)}\\
&\leq \left(8+\frac{24}{(\frac{d+s}{2}-1)}\right)\left[\left(\frac{9D_{\mathcal{X},L}^2}{\abs{\Sigma}m}\right)^{\frac{1}{d+s}}+\left(\frac{9D_{\mathcal{X},L}^2}{\abs{\Sigma}n}\right)^{\frac{1}{d+s}}\right]\\
&\quad+ \left(\frac{24}{\sqrt{m}}+ \frac{24}{\sqrt{n}}\right)\int_{\delta^*}^{M}\sqrt{\ln\mathcal{N}(\mathcal{F}_0,\delta,\norm{\cdot}_{\infty})}\diff{\delta}.
\end{align*}

For $d=1$, the first integral in \eqref{eq:dudleybound} in the one-dimensional case does not have a singularity at $\alpha = 0$. On the other hand, the covering number $\mathcal{N}(\mathcal{F}_0,\delta,\norm{\cdot}_{\infty})$ is bounded by the covering number for which we extend the domain to an interval that contains $\mathcal{X}_0$. Replacing the interval $[0,1]$ by an interval of length $\text{diam}(\mathcal{X}_0)$ in Lemma 5.16 in \cite{van2014probability}, there exists a constant $c>0$ such that 
    \[
    \mathcal{N}(\mathcal{F}_0,\delta,\norm{\cdot}_{\infty})\leq e^{\frac{cL\cdot\text{diam}(\mathcal{X}_0)}{\delta}}\,\,\, \text{for}\,\, \delta<M=L\cdot\text{diam}(\mathcal{X}_0).\]
Therefore, we have 
\begin{align*}
    8\alpha+\frac{24}{\sqrt{m}}\int_{\alpha}^{M}\sqrt{\ln\mathcal{N}(\mathcal{F}_0,\delta,\norm{\cdot}_{\infty})}\diff{\delta} \leq 8\alpha + \frac{24}{\sqrt{m}}\int_{\alpha}^{M}\sqrt{\frac{cL\cdot\text{diam}(\mathcal{X}_0)}{\delta}}\diff{\delta},
\end{align*}
whose minimum is achieved at $\alpha = \frac{9cL\cdot\text{diam}(\mathcal{X}_0)}{m}$. This implies that
\begin{align*}
\inf_{\alpha>0} 8\alpha+\frac{24}{\sqrt{m}}\int_{\alpha}^{M}\sqrt{\ln\mathcal{N}(\mathcal{F}_0,\delta,\norm{\cdot}_{\infty})}\diff{\delta} &\leq \frac{72cL\cdot\text{diam}(\mathcal{X}_0)}{m} + \frac{48L\sqrt{c}\cdot\text{diam}(\mathcal{X}_0)}{\sqrt{m}} - \frac{144cL\cdot\text{diam}(\mathcal{X}_0)}{m}\\
&= \frac{48L\sqrt{c}\cdot\text{diam}(\mathcal{X}_0)}{\sqrt{m}} - \frac{72cL\cdot\text{diam}(\mathcal{X}_0)}{m}.
\end{align*}
Hence, we have
\begin{align*}
&E_{X,Y}\sup_{\gamma\in\text{Lip}_L(\mathcal{X}_0)}\abs{E_{P_{\mathcal{X}_0}}[\gamma]-\frac{1}{m}\sum_{i=1}^m\gamma(x_i)-\left(E_{Q_{\mathcal{X}_0}}[\gamma]-\frac{1}{n}\sum_{i=1}^n\gamma(y_i)\right)}\\
&\quad\leq \frac{48L\sqrt{c}\cdot\text{diam}(\mathcal{X}_0)}{\sqrt{m}} - \frac{72cL\cdot\text{diam}(\mathcal{X}_0)}{m} + \frac{48L\sqrt{c}\cdot\text{diam}(\mathcal{X}_0)}{\sqrt{n}} - \frac{72cL\cdot\text{diam}(\mathcal{X}_0)}{n}.
\end{align*}
Finally, by a simple change of variable for the probability provided in \eqref{eq:mcdiarmid}, we prove the theorem.
\end{proof}

\begin{remark}\label{rmk:wss-1d}
Though we do not directly observe the effect under the group invariance in the case when $d=1$ in \cref{thm:gammaIPM}, the upper bound can be improved in some special cases. Here we analyze \cref{example:wss1d} as an example. Replacing the interval $[0,1]$ by $\mathcal{X}_0= [0,\frac{1}{\abs{\Sigma}})$ in Lemma 5.16 in \cite{van2014probability}, there exists a constant $c>0$ such that 
    \[
    \mathcal{N}(\mathcal{F}_0,\delta,\norm{\cdot}_{\infty})\leq e^{\frac{cL}{\abs{\Sigma}\delta}}\,\,\, \text{for}\,\, \delta<M=L\cdot\text{diam}(\mathcal{X}_0).\]
    Therefore, we have 
\begin{align*}
    8\alpha+\frac{24}{\sqrt{m}}\int_{\alpha}^{M}\sqrt{\ln\mathcal{N}(\mathcal{F}_0,\delta,\norm{\cdot}_{\infty})}\diff{\delta} = 8\alpha + \frac{24}{\sqrt{m}}\int_{\alpha}^{M}\sqrt{\frac{cL}{\abs{\Sigma}\delta}}\diff{\delta},
\end{align*}
whose minimum is achieved at $\alpha = \frac{9cL}{m\abs{\Sigma}}$. This implies that
\begin{align*}
\inf_{\alpha>0} 8\alpha+\frac{24}{\sqrt{m}}\int_{\alpha}^{M}\sqrt{\ln\mathcal{N}(\mathcal{F}_0,\delta,\norm{\cdot}_{\infty})}\diff{\delta} = \frac{72cL}{\abs{\Sigma}m} + \frac{48L\sqrt{c}}{\abs{\Sigma}\sqrt{m}} - \frac{144cL}{\abs{\Sigma}m}
= \frac{48L\sqrt{c}}{\abs{\Sigma}\sqrt{m}} - \frac{72cL}{\abs{\Sigma}m}.
\end{align*}
Hence, we have
\begin{align*}
E_{X,Y}\sup_{\gamma\in\text{Lip}_L(\mathcal{X}_0)}\abs{E_{P_{\mathcal{X}_0}}[\gamma]-\frac{1}{m}\sum_{i=1}^m\gamma(x_i)-\left(E_{Q_{\mathcal{X}_0}}[\gamma]-\frac{1}{n}\sum_{i=1}^n\gamma(y_i)\right)}
\leq \frac{48L\sqrt{c}}{\abs{\Sigma}\sqrt{m}} - \frac{72cL}{\abs{\Sigma}m} + \frac{48L\sqrt{c}}{\abs{\Sigma}\sqrt{n}} - \frac{72cL}{\abs{\Sigma}n}.
\end{align*}
This matches the numerical result in \cref{fig:wasserstein} where the ratio curves are around $4$, since our group sizes are $\abs{\Sigma} = 1, 4, 16,64,256$, increasing by a factor of $4$,
\end{remark}

\subsection{$(f_{\alpha},\Gamma)$-divergence}\label{appendix:falphagamma}
We assume Assumption~\ref{assumption:sapiro} also holds in this case.
\begin{theorem}\label{thm:f-gamma}
Let $\mathcal{X} = \Sigma\times\mathcal{X}_0$ be a subset of $\mathbb{R}^D$ equipped with the Euclidean distance, $f(x)=f_\alpha(x)=\frac{x^\alpha-1}{\alpha(\alpha-1)}$, $\alpha>1$ and $\Gamma = \text{Lip}_L(\mathcal{X})$. Assume that $\mathcal{N}(\mathcal{X},\delta)\lesssim \delta^{-d}$ for sufficiently small $\delta$. Suppose $P$ and $Q$ are $\Sigma$-invariant distributions on $\mathcal{X}$. We have

1) if $d\geq 2$, then for any $s>0$ and $m,n$ sufficiently large, we have with probability at least $1-\epsilon$,
\begin{align*}
\abs{D_{f_\alpha}^{\Gamma}(P\| Q)-D_{f_\alpha}^{\Gamma_{\Sigma}}(P_m\|Q_n)}&\leq\left(8+\frac{24}{(\frac{d+s}{2}-1)}\right)\left[\left(\frac{9D_{\mathcal{X},L}^2}{\abs{\Sigma}m}\right)^{\frac{1}{d+s}}+\left(\frac{9D_{\mathcal{X},L'}^2}{\abs{\Sigma}n}\right)^{\frac{1}{d+s}}\right]\\
&\quad+ \frac{24\bar{D}_{\mathcal{X}_0, L}}{\sqrt{m}} + \frac{24\bar{D}_{\mathcal{X}_0, L'}}{\sqrt{n}}\\
&\quad + \sqrt{\frac{2(M_1^2m+M_0^2n)}{mn}\ln \frac{1}{\epsilon}},
\end{align*}
where $D_{\mathcal{X}, L}$ depends only on $\mathcal{X}$ and $L$, and $D_{\mathcal{X}, L'}$ depends only on $\mathcal{X}$, $L$ and $\alpha$; $\bar{D}_{\mathcal{X}_0, L}$ depends only on $\mathcal{X}_0$ and $L$, and $\bar{D}_{\mathcal{X}_0, L'}$ depends only on $\mathcal{X}_0$ and $L$ and $\alpha$, and both are increasing in $\mathcal{X}_0$; $M_0$ and $M_1$ both only depend on $\mathcal{X}$, $L$ and $\alpha$;

2) if $d=1$, for any $\epsilon>0$ and $m,n$ sufficiently large, we have with probability at least $1-\epsilon$,
\begin{align*}
\abs{D_{f_\alpha}^{\Gamma}(P\| Q)-D_{f_\alpha}^{\Gamma_{\Sigma}}(P_m\|Q_n)}&\leq\frac{48L\sqrt{c}\cdot\text{\normalfont diam}(\mathcal{X}_0)}{\sqrt{m}} - \frac{72cL\cdot\text{\normalfont diam}(\mathcal{X}_0)}{m} + \frac{48L'\sqrt{c}\cdot\text{\normalfont diam}(\mathcal{X}_0)}{\sqrt{n}} - \frac{72cL'\cdot\text{\normalfont diam}(\mathcal{X}_0)}{n}\\
&\quad+\sqrt{\frac{2(M_1^2m+M_0^2n)}{mn}\ln \frac{1}{\epsilon}},
\end{align*}
where $c>0$ is an absolute constant independent of $\mathcal{X}_0$; $L'$ depends only on $\CX$, $L$ and $\alpha$; $M_0$ and $M_1$ both only depend on $\mathcal{X}$, $L$ and $\alpha$.
\end{theorem}
Before proving this theorem, we first provide the following lemma.

\begin{lemma}\label{lemma:falphainfinitybound}
    $D_{f_\alpha}^{\Gamma}(P\| Q) = D_{f_\alpha}^{\mathcal{F}}(P\| Q)$, where $$\mathcal{F} = \left\{\gamma\in\text{Lip}_{L}(\mathcal{X}):\norm{\gamma}_\infty\leq (\alpha-1)^{-1}+L\cdot\text{diam}(\mathcal{X}) \right\},$$
    and $P$ and $Q$ are probability distributions on $\mathcal{X}$ that are not necessarily $\Sigma$-invariant.
\end{lemma}
\begin{proof}
    For any fixed $\gamma\in\Gamma$, let $h(\nu) = E_{P}[\gamma+\nu]-E_{Q}[f_\alpha^*(\gamma+\nu)]$.
    We know that $\sup_{x\in\mathcal{X}}\gamma(x)- \inf_{x\in\mathcal{X}}\gamma(x)\leq L\cdot\text{diam}(\mathcal{X})$, 
    so interchanging the integration with differentiation is allowed by the dominated convergence theorem: $h'(\nu) = 1 -E_Q[f_\alpha^{*\prime}(\gamma+\nu)]$, where
    \begin{align*}
    f_\alpha^{*\prime}(y) = (\alpha-1)^{\frac{1}{\alpha-1}}y^{\frac{1}{\alpha-1}}\mathbf{1}_{y>0}.
    \end{align*}
    If $\inf_{x\in\mathcal{X}}\gamma(x)>(\alpha-1)^{-1}$, then $h'(0)<0$. So there exists some $\nu_0<0$ such that $E_{P}[\gamma+\nu_0]-E_{Q}[f_\alpha^*(\gamma+\nu_0)] = h(\nu_0)>h(0) = E_{P}[\gamma]-E_{Q}[f_\alpha^*(\gamma)]$. This indicates the supremum in $D_{f}^{\Gamma}(P\| Q)$ is attained only if $\sup_{x\in\mathcal{X}}\gamma(x)\leq (\alpha-1)^{-1}+L\cdot\text{diam}(\mathcal{X})$. On the other hand, if $\sup_{x\in\mathcal{X}}\gamma(x)<0$, then there exists $\nu_0>0$ that satisfies $\sup_{x\in\mathcal{X}}\gamma(x)+\nu_0<0$ such that $E_{P}[\gamma+\nu_0]-E_{Q}[f_\alpha^*(\gamma+\nu_0)] = E_{P}[\gamma]+\nu_0> E_{P}[\gamma] = E_{P}[\gamma]-E_{Q}[f_\alpha^*(\gamma)]$. This indicates that the supremum in $D_{f}^{\Gamma}(P\| Q)$ is attained only if $\inf_{x\in\mathcal{X}}\gamma(x)\geq -L\cdot\text{diam}(\mathcal{X})$. Therefore, we have that the supremum in $D_{f}^{\Gamma}(P\| Q)$ is attained only if $\norm{\gamma}_\infty\leq(\alpha-1)^{-1}+L\cdot\text{diam}(\mathcal{X})$.
\end{proof}

\begin{proof}[Proof of \cref{thm:f-gamma}]
Similar to the beginning of the proof of Theorem \ref{thm:gammaIPM}, we have by Lemma \ref{lemma:falphainfinitybound} that
\begin{align*}\label{eq:alphasup}
    &\abs{D_{f_\alpha}^{\Gamma}(P\| Q)-D_{f_\alpha}^{\Gamma_\Sigma}(P_m,Q_n)}\\
    &= \abs{\sup_{\substack{\gamma\in\Gamma_{\Sigma}\\ \norm{\gamma}_\infty\leq M_0}}\left\{E_{P}[\gamma]-E_{Q}[f_\alpha^*(\gamma)]\right\} - \sup_{\substack{\gamma\in\Gamma_{\Sigma}\\ \norm{\gamma}_\infty\leq M_0}}\left\{E_{P_m}[\gamma]-E_{Q_n}[f_\alpha^*(\gamma)]\right\}}\\
    &\leq\sup_{\substack{\gamma\in\Gamma_{\Sigma}\\ \norm{\gamma}_\infty\leq M_0}}\abs{E_P[\gamma]-\frac{1}{m}\sum_{i=1}^m\gamma(x_i)-\left(E_Q[f_\alpha^*(\gamma)]-\frac{1}{n}\sum_{i=1}^n f_\alpha^*\left(\gamma(y_i)\right)\right)}\\
    &=\sup_{\substack{\gamma\in\Gamma_{\Sigma}\\ \norm{\gamma}_\infty\leq M_0}}\abs{E_P[\gamma]-\frac{1}{m}\sum_{i=1}^m\gamma\left(T_0(x_i)\right)-\left(E_Q[f_\alpha^*(\gamma)]-\frac{1}{n}\sum_{i=1}^n f_\alpha^*\left(\gamma(T_0(y_i))\right)\right)}\\
    &\leq \sup_{\substack{\gamma\in\text{Lip}_L(\mathcal{X}_0)\\ \norm{\gamma}_\infty\leq M_0}}\abs{E_{P_{\mathcal{X}_0}}[\gamma]-\frac{1}{m}\sum_{i=1}^m\gamma(T_0(x_i))-\left(E_{Q_{\mathcal{X}_0}}[f_\alpha^*(\gamma)]-\frac{1}{n}\sum_{i=1}^n f_\alpha^*\left(\gamma(T_0(y_i))\right)\right)}\\
    &:=g(x_1,\dots,x_m,y_1,\dots,y_n),
\end{align*}
where $T_0$ is the same as defined in \eqref{def:quotientmap}.
The denominator in the exponent when applying the McDiarmid’s inequality is thus equal to
\[
m\left(\frac{2M_0}{m}\right)^2+n\left(\frac{2M_1}{n}\right)^2 = \frac{4M_0^2}{m}+\frac{4M_1^2}{n},
\]
where $M_0=(\alpha-1)^{-1}+L\cdot\text{diam}(\mathcal{X})$, $M_1 = f_\alpha^*(M_0)$, since for any $\gamma$ such that $\norm{\gamma}_\infty\leq M_0$, we have $\norm{f_\alpha^*\circ\gamma}_\infty\leq M_1$. Denoting by $X' = \{x'_1,x'_2,\dots,x'_m\}$ and $Y' = \{y'_1,y'_2,\dots,y'_n\}$ the i.i.d. samples drawn from $P_{\mathcal{X}_0}$ and $Q_{\mathcal{X}_0}$. Also note that $T_0(x_1),\dots,T_0(x_m)$ and $T_0(y_1),\dots,T_0(y_n)$ can be viewed as i.i.d. samples on $\mathcal{X}_0$ drawn from $P_{\mathcal{X}_0}$ and $Q_{\mathcal{X}_0}$ respectively, such that the expectation
\begin{align*}
&E_{X,Y}g(x_1,x_2,\dots,x_m,y_1,y_2,\dots,y_n)\\
&= E_{X,Y}\sup_{\substack{\gamma\in\text{Lip}_L(\mathcal{X}_0)\\ \norm{\gamma}_\infty\leq M_0}}\abs{E_{P_{\mathcal{X}_0}}[\gamma]-\frac{1}{m}\sum_{i=1}^m\gamma(T_0(x_i))-\left(E_{Q_{\mathcal{X}_0}}[f_\alpha^*(\gamma)]-\frac{1}{n}\sum_{i=1}^n f_\alpha^*\left(\gamma(T_0(y_i))\right)\right)}
\end{align*}
can be replaced by the equivalent quantity
\[
E_{X,Y}\sup_{\substack{\gamma\in\text{Lip}_L(\mathcal{X}_0)\\ \norm{\gamma}_\infty\leq M_0}}\abs{E_{P_{\mathcal{X}_0}}[\gamma]-\frac{1}{m}\sum_{i=1}^m\gamma(x_i)-\left(E_{Q_{\mathcal{X}_0}}[f_\alpha^*(\gamma)]-\frac{1}{n}\sum_{i=1}^n f_\alpha^*\left(\gamma(y_i)\right)\right)},
\]
where $X = \{x_1,x_2,\dots,x_m\}$ and $Y =\{y_1,y_2,\dots,y_n\}$ are are i.i.d. samples on $\mathcal{X}_0$ drawn from $P_{\mathcal{X}_0}$ and $Q_{\mathcal{X}_0}$ respectively. Then we have
\begin{align*}
&E_{X,Y}\sup_{\substack{\gamma\in\text{Lip}_L(\mathcal{X}_0)\\ \norm{\gamma}_\infty\leq M_0}}\abs{E_{P_{\mathcal{X}_0}}[\gamma]-\frac{1}{m}\sum_{i=1}^m\gamma(x_i)-\left(E_{Q_{\mathcal{X}_0}}[f_\alpha^*(\gamma)]-\frac{1}{n}\sum_{i=1}^n f_\alpha^*\left(\gamma(y_i)\right)\right)}\\
&= E_{X,Y}\sup_{\substack{\gamma\in\text{Lip}_L(\mathcal{X}_0)\\ \norm{\gamma}_\infty\leq M_0}}\abs{E_{X'}\left(\frac{1}{m}\sum_{i=1}^m\gamma(x'_i)\right)-\frac{1}{m}\sum_{i=1}^m\gamma(x_i)-E_{Y'}\left(\frac{1}{n}\sum_{i=1}^n f_\alpha^*\left(\gamma(y_i')\right)\right)+\frac{1}{n}\sum_{i=1}^nf_\alpha^*\left(\gamma(y_i)\right)}\\
&\leq E_{X,Y,X',Y'}\sup_{\substack{\gamma\in\text{Lip}_L(\mathcal{X}_0)\\ \norm{\gamma}_\infty\leq M_0}}\abs{\frac{1}{m}\sum_{i=1}^m\gamma(x'_i)-\frac{1}{m}\sum_{i=1}^m\gamma(x_i)-\frac{1}{n}\sum_{i=1}^n f_\alpha^*\left(\gamma(y_i')\right)+\frac{1}{n}\sum_{i=1}^n f_\alpha^*\left(\gamma(y_i)\right)}\\
&=E_{X,Y,X',Y',\xi,\xi'}\sup_{\substack{\gamma\in\text{Lip}_L(\mathcal{X}_0)\\ \norm{\gamma}_\infty\leq M_0}}\abs{\frac{1}{m}\sum_{i=1}^m\xi_i\left(\gamma(x'_i)-\gamma(x_i)\right)-\frac{1}{n}\sum_{i=1}^n\xi'_i\left(f_\alpha^*\left(\gamma(y_i')\right)-f_\alpha^*\left(\gamma(y_i)\right)\right)}\\
&\leq E_{X,X',\xi}\sup_{\substack{\gamma\in\text{Lip}_L(\mathcal{X}_0)\\ \norm{\gamma}_\infty\leq M_0}}\abs{\frac{1}{m}\sum_{i=1}^m\xi_i\left(\gamma(x'_i)-\gamma(x_i)\right)} + E_{Y,Y',\xi'}\sup_{\substack{\gamma\in\text{Lip}_L(\mathcal{X}_0)\\ \norm{\gamma}_\infty\leq M_0}}\abs{\frac{1}{n}\sum_{i=1}^n\xi'_i\left(f_\alpha^*\left(\gamma(y_i')\right)-f_\alpha^*\left(\gamma(y_i)\right)\right)}\\
&\leq \inf_{\alpha>0} 8\alpha+\frac{24}{\sqrt{m}}\int_{\alpha}^{M_0}\sqrt{\ln\mathcal{N}(\mathcal{F}_0,\delta,\norm{\cdot}_{\infty})}\diff{\delta} + \inf_{\alpha>0} 8\alpha+\frac{24}{\sqrt{n}}\int_{\alpha}^{M_1}\sqrt{\ln\mathcal{N}(\mathcal{F}_1,\delta,\norm{\cdot}_{\infty})}\diff{\delta},
\end{align*}

where $\mathcal{F}_0 = \left\{\gamma\in\text{Lip}_{L}(\mathcal{X}_0):\norm{\gamma}_\infty\leq M_0 \right\}$ and $\mathcal{F}_1 = \left\{\gamma\in\text{Lip}_{L'}(\mathcal{X}_0):\norm{\gamma}_\infty\leq M_1 \right\}$, since for any $\gamma\in\mathcal{F}_0$, $\norm{f_\alpha^*\circ\gamma}_\infty\leq M_1$ and $\abs{\frac{\diff}{\diff{y}}f_\alpha^*(y)}\leq(\alpha-1)^{\frac{1}{\alpha-1}}(M_0)^{\frac{1}{\alpha-1}}$ for $\abs{y}\leq M_0$ such that $f_\alpha^*\circ\gamma$ is $L'$-Lipschitz, where $M_1 = f_\alpha^*(M_0)$ and $L' = L(\alpha-1)^{\frac{1}{\alpha-1}}(M_0)^{\frac{1}{\alpha-1}}$. Then the rest of the proof follows from the proof of \cref{thm:gammaIPM}.
\end{proof}

\subsection{MMD}\label{appendix:MMD}
We assume the kernel $k(x,y)$ satisfies Assumption~\ref{assump:kernel}. Furthermore, let $\phi(x)$ be the evaluation functional at $x$ in $\mathcal{H}$: $\langle \phi(x),\phi(y)\rangle_\mathcal{H} = k(x,y),\forall x,y\in\mathcal{H}$.

\begin{theorem}\label{thm:mmd}
Let $\mathcal{X} = \Sigma\times\mathcal{X}_0$ and $\mathcal{H}$ be a RKHS on $\mathcal{X}$ whose kernel satisfies Assumption~\ref{assump:kernel}. Suppose $P$ and $Q$ are $\Sigma$-invariant distributions on $\mathcal{X}$. Then for $m,n$ sufficiently large and any $\epsilon>0$ we have with probability at least $1-\epsilon$,
\begin{align*}
    \abs{\text{MMD}(P,Q)-\text{MMD}^{\Sigma}(P_m,Q_n)}
    &< 2K^{\frac{1}{2}}\left[1+c(\abs{\Sigma}-1)\right]^{\frac{1}{2}}\left(\frac{1}{\sqrt{\abs{\Sigma}m}} + \frac{1}{\sqrt{\abs{\Sigma}n}}\right) \\
    &\quad+\sqrt{\frac{2K(1+c(\abs{\Sigma}-1))\ln(\frac{1}{\epsilon})}{\abs{\Sigma}}}\sqrt{\frac{1}{m}+\frac{1}{n}},
\end{align*}
where $K$ and $c$ are the constants in Assumption~\ref{assump:kernel}.
\end{theorem}
Before proving the theorem, we provide the following lemma.
\begin{lemma}\label{lemma:mmd-Rademacher}
    Suppose the kernel in an RKHS satisfies Assumption~\ref{assump:kernel}, and $\xi=\{\xi_1,\dots,\xi_m\}$ is a set of independent random variables, each of which takes values on $\{-1,1\}$ with equal probabilities. Then we have
    \begin{align*}  E_{\xi}\sup_{\norm{\gamma}_{\mathcal{H}}\leq 1}\abs{\frac{1}{m\abs{\Sigma}}\sum_{i=1}^m\xi_i\sum_{j=1}^{\abs{\Sigma}}\gamma(\sigma_j x_i)}\leq\frac{\left(1+c(\abs{\Sigma}-1)\right)K^{\frac{1}{2}}}{\sqrt{\abs{\Sigma}m}}.
    \end{align*}
\end{lemma}
\begin{proof}
Since the witness function to attain the supremum is explicit, we can write
    \begin{align*}     E_{\xi}\sup_{\norm{\gamma}_{\mathcal{H}}\leq 1}\abs{\frac{1}{m\abs{\Sigma}}\sum_{i=1}^m\xi_i\sum_{j=1}^{\abs{\Sigma}}\gamma(\sigma_jx_i)}&= E_{\xi}\norm{\frac{1}{m\abs{\Sigma}}\sum_{i=1}^m\xi_i\sum_{j=1}^{\abs{\Sigma}}\phi(\sigma_jx_i)}_{\mathcal{H}}\\
    &=\frac{1}{m\abs{\Sigma}}E_{\xi}\left[\sum_{i,i'=1}^m\xi_i\xi_{i'}\sum_{j,j'=1}^{\abs{\Sigma}}k(\sigma_jx_i,\sigma_{j'}x_{i'})]\right]^{\frac{1}{2}}\\
    &\leq\frac{1}{m\abs{\Sigma}}\left[E_{\xi}\sum_{i,i'=1}^m\xi_i\xi_{i'}\sum_{j,j'=1}^{\abs{\Sigma}}k(\sigma_jx_i,\sigma_{j'}x_{i'})]\right]^{\frac{1}{2}}\\
    &=\frac{1}{m\abs{\Sigma}}\left[E_{\xi}\sum_{i=1}^m(\xi_i)^2\sum_{j,j'=1}^{\abs{\Sigma}}k(\sigma_jx_i,\sigma_{j'}x_{i})]\right]^{\frac{1}{2}}\\
    &\leq\frac{1}{m\abs{\Sigma}}\left[m\cdot\left(\abs{\Sigma}K+c(\abs{\Sigma}^2-\abs{\Sigma})K\right)\right]^{\frac{1}{2}}\\
    &= \frac{K^{\frac{1}{2}}\left[1+c(\abs{\Sigma}-1)\right]^{\frac{1}{2}}}{\sqrt{\abs{\Sigma}m}}.
    \end{align*}
\end{proof}

\begin{proof}[Proof of \cref{thm:mmd}]
The proof below is a generalization of the proof of Theorem 7 in \cite{gretton2012kernel}, which does not need the notion of covering numbers due to the structure of RKHS.
\begin{align*}
    &\abs{\text{MMD}(P,Q)-\text{MMD}^{\Sigma}(P_m,Q_n)}\\
    &= \abs{\text{MMD}(P,Q)-\text{MMD}(S^{\Sigma}[P_m],S^{\Sigma}[Q_n])}\\
    &= \abs{\sup_{\norm{\gamma}_{\mathcal{H}}\leq 1}\{E_P[\gamma]-E_Q[\gamma]\} - \sup_{\norm{\gamma}_{\mathcal{H}}\leq 1}\{E_{S^{\Sigma}[P_m]}[\gamma]-E_{S^{\Sigma}[Q_n]}[\gamma]\}}\\
    &= \abs{\sup_{\norm{\gamma}_{\mathcal{H}}\leq 1}\{E_P[\gamma]-E_Q[\gamma]\} - \sup_{\norm{\gamma}_{\mathcal{H}}\leq 1}\{\frac{1}{m\abs{\Sigma}}\sum_{i=1}^m\sum_{j=1}^{\abs{\Sigma}}\gamma(\sigma_j x_i)-\frac{1}{n\abs{\Sigma}}\sum_{i=1}^n\sum_{j=1}^{\abs{\Sigma}}\gamma(\sigma_j y_i)\}}\\
    &\leq \sup_{\norm{\gamma}_{\mathcal{H}}\leq 1}\abs{E_P[\gamma]-E_Q[\gamma]-\frac{1}{m\abs{\Sigma}}\sum_{i=1}^m\sum_{j=1}^{\abs{\Sigma}}\gamma(\sigma_j x_i)+\frac{1}{n\abs{\Sigma}}\sum_{i=1}^n\sum_{j=1}^{\abs{\Sigma}}\gamma(\sigma_j y_i)}\\
    &:=f(x_1,x_2,\dots,x_m,y_1,y_2,\dots,y_n).
\end{align*}
Now we estimate the upper bound of the difference of $f$ if we change one of $x_i$'s.

\begin{align}\label{MMD:McDiarmid}
&\abs{f(x_1,\dots,x_i,\dots,y_1,\dots,y_n)-f(x_1,\dots,\tilde{x}_i,\dots,y_1,\dots,y_n)}\nonumber\\
&\leq \sup_{\norm{\gamma}_{\mathcal{H}}\leq 1}\abs{\frac{1}{m\abs{\Sigma}}\sum_{j=1}^{\abs{\Sigma}}\gamma(\sigma_j x_i) - \gamma(\sigma_j \tilde{x}_i)}\nonumber\\
&= \frac{1}{m\abs{\Sigma}} \norm{\sum_{j=1}^{\abs{\Sigma}}\phi(\sigma_j x_i) - \phi(\sigma_j \tilde{x}_i)}_{\mathcal{H}}\\
&\leq \frac{1}{m\abs{\Sigma}}\left(\norm{\sum_{j=1}^{\abs{\Sigma}}\phi(\sigma_j x_i) }_{\mathcal{H}} + \norm{\sum_{j=1}^{\abs{\Sigma}}\phi(\sigma_j \tilde{x}_i) }_{\mathcal{H}}\right).\nonumber
\end{align}

To bound $\norm{\sum_{j=1}^{\abs{\Sigma}}\phi(\sigma_j x_i) }_{\mathcal{H}}$, we have
\begin{align*}
\norm{\sum_{j=1}^{\abs{\Sigma}}\phi(\sigma_j x_i) }_{\mathcal{H}} &= \left[\sum_{j=1}^{\abs{\Sigma}}k(\sigma_j x_i,\sigma_j x_i) + \sum_{j\neq l}k(\sigma_j x_i,\sigma_l x_i)\right]^{\frac{1}{2}}\\
&= \left[\sum_{j=1}^{\abs{\Sigma}}k(\sigma_j x_i,\sigma_j x_i) + \sum_{\sigma_j\neq id}k(\sigma_j x_i, x_i)\right]^{\frac{1}{2}}\\
&\leq \left[\abs{\Sigma}\cdot K+\left(\abs{\Sigma}^2-\abs{\Sigma}\right)\cdot cK\right]^{\frac{1}{2}}.
\end{align*}
The upper bound of the difference of $f$ if we change one of $y_i$'s can be derived in the same way. To apply the McDiarmid’s inequality, the denominator in the exponent is thus
\begin{align*}
    &m\cdot\frac{4\left[\abs{\Sigma}\cdot K+\left(\abs{\Sigma}^2-\abs{\Sigma}\right)\cdot cK\right]}{m^2\abs{\Sigma}^2}+n\cdot\frac{4\left[\abs{\Sigma}\cdot K+\left(\abs{\Sigma}^2-\abs{\Sigma}\right)\cdot cK\right]}{n^2\abs{\Sigma}^2}\\
    &\leq 4K(\frac{1}{m}+\frac{1}{n})\cdot\frac{1+c(\abs{\Sigma}-1)}{\abs{\Sigma}}.
\end{align*}
Moreover, we can extend inequality (16) in \cite{gretton2012kernel} to take into account the group invariance. Denoting by $X' = \{x'_1,x'_2,\dots,x'_m\}$ and $Y' = \{y'_1,y'_2,\dots,y'_n\}$ the i.i.d. samples drawn from $P$ and $Q$, and $\xi=\{\xi_1,\dots,\xi_m\}$, $\xi'=\{\xi'_1,\dots,\xi'_n\}$ sets of independent random variables, each of which takes values on $\{-1,1\}$ with equal probabilities, we have
\begin{align*}
    &E_{X,Y}f(x_1,x_2,\dots,x_m,y_1,y_2,\dots,y_n)\\
    &=E_{X,Y}\sup_{\norm{\gamma}_{\mathcal{H}}\leq 1}\abs{E_P[\gamma]-E_Q[\gamma]-\frac{1}{m\abs{\Sigma}}\sum_{i=1}^m\sum_{j=1}^{\abs{\Sigma}}\gamma(\sigma_j x_i)+\frac{1}{n\abs{\Sigma}}\sum_{i=1}^n\sum_{j=1}^{\abs{\Sigma}}\gamma(\sigma_j y_i)}\\
    &=E_{X,Y}\sup_{\norm{\gamma}_{\mathcal{H}}\leq 1}\left|E_{X'}\left(\frac{1}{m\abs{\Sigma}}\sum_{i=1}^m\sum_{j=1}^{\abs{\Sigma}}\gamma(\sigma_j x'_i)\right)-E_{Y'}\left(\frac{1}{n\abs{\Sigma}}\sum_{i=1}^n\sum_{j=1}^{\abs{\Sigma}}\gamma(\sigma_j y'_i)\right)\right.\\
    & \quad\quad\quad\quad\quad\quad\quad\quad\quad\quad\quad\quad\quad\quad\quad\quad\quad\quad\quad\quad \left.-\frac{1}{m\abs{\Sigma}}\sum_{i=1}^m\sum_{j=1}^{\abs{\Sigma}}\gamma(\sigma_j x_i)+\frac{1}{n\abs{\Sigma}}\sum_{i=1}^n\sum_{j=1}^{\abs{\Sigma}}\gamma(\sigma_j y_i)\right|\\
    &\leq E_{X,Y,X',Y'}\sup_{\norm{\gamma}_{\mathcal{H}}\leq 1}\abs{\frac{1}{m\abs{\Sigma}}\sum_{i=1}^m\sum_{j=1}^{\abs{\Sigma}}\left(\gamma(\sigma_j x'_i)-\gamma(\sigma_j x_i)\right)-\frac{1}{n\abs{\Sigma}}\sum_{i=1}^n\sum_{j=1}^{\abs{\Sigma}}\left(\gamma(\sigma_j y'_i)-\gamma(\sigma_j y_i)\right)}\\
    &=E_{X,Y,X',Y',\xi,\xi'}\sup_{\norm{\gamma}_{\mathcal{H}}\leq 1}\abs{\frac{1}{m\abs{\Sigma}}\sum_{i=1}^m\xi_i\sum_{j=1}^{\abs{\Sigma}}\left(\gamma(\sigma_j x'_i)-\gamma(\sigma_j x_i)\right)-\frac{1}{n\abs{\Sigma}}\sum_{i=1}^n\xi'_i\sum_{j=1}^{\abs{\Sigma}}\left(\gamma(\sigma_j y'_i)-\gamma(\sigma_j y_i)\right)}\\
    &\leq E_{X,X',\xi}\sup_{\norm{\gamma}_{\mathcal{H}}\leq 1}\abs{\frac{1}{m\abs{\Sigma}}\sum_{i=1}^m\xi_i\sum_{j=1}^{\abs{\Sigma}}\left(\gamma(\sigma_j x'_i)-\gamma(\sigma_j x_i)\right)}+E_{Y,Y',\xi'}\sup_{\norm{\gamma}_{\mathcal{H}}\leq 1}\abs{\frac{1}{n\abs{\Sigma}}\sum_{i=1}^n\xi'_i\sum_{j=1}^{\abs{\Sigma}}\left(\gamma(\sigma_j y'_i)-\gamma(\sigma_j y_i)\right)}\\
    &\leq 2\left[\frac{K^{\frac{1}{2}}\left[1+c(\abs{\Sigma}-1)\right]^{\frac{1}{2}}}{\sqrt{\abs{\Sigma}m}} + \frac{K^{\frac{1}{2}}\left[1+c(\abs{\Sigma}-1)\right]^{\frac{1}{2}}}{\sqrt{\abs{\Sigma}n}}\right],
\end{align*}
where the last inequality is due to Lemma \ref{lemma:mmd-Rademacher}. Therefore, by the McDiarmid’s theorem, we have
\begin{align*}
    &\mathbb{P}\left(\abs{\text{MMD}(P,Q)-\text{MMD}^{\Sigma}(P_m,Q_n)} - 2K^{\frac{1}{2}}\left[1+c(\abs{\Sigma}-1)\right]^{\frac{1}{2}}\left(\frac{1}{\sqrt{\abs{\Sigma}m}} + \frac{1}{\sqrt{\abs{\Sigma}n}}\right)>\epsilon\right)\\
    &\qquad\leq\exp\left(-\frac{\epsilon^2mn\abs{\Sigma}}{2K(m+n)(1+c(\abs{\Sigma}-1))}\right).
\end{align*}
By a change of variable, we have with probability at least $1-\epsilon$, 
\begin{align*}
    \abs{\text{MMD}(P,Q)-\text{MMD}^{\Sigma}(P_m,Q_n)}
    &< 2K^{\frac{1}{2}}\left[1+c(\abs{\Sigma}-1)\right]^{\frac{1}{2}}\left(\frac{1}{\sqrt{\abs{\Sigma}m}} + \frac{1}{\sqrt{\abs{\Sigma}n}}\right) \\
    &\quad+\sqrt{\frac{2K(1+c(\abs{\Sigma}-1))\ln(\frac{1}{\epsilon})}{\abs{\Sigma}}}\sqrt{\frac{1}{m}+\frac{1}{n}}.
\end{align*}
\end{proof}

\section*{Acknowledgements}
The research of M.K. and L.R.-B. was partially supported by  the Air Force Office of Scientific Research (AFOSR) under the grant FA9550-21-1-0354.
The research of M. K. and L.R.-B. was partially supported by the National Science Foundation (NSF) under the grants DMS-2008970 and TRIPODS CISE-1934846. The research of Z.C and W.Z. was  partially supported by NSF under DMS-2052525 and DMS-2140982. We thank Yulong Lu for the insightful discussions.

\bibliographystyle{siam}
\bibliography{mybibfile.bib}

\end{document}